\documentclass[a4paper,11pt,reqno]{amsart}

\usepackage{graphicx}
\usepackage{xcolor}
\usepackage{xcolor}
\definecolor{darkgreen}{rgb}{0,0.41,0.11}
\usepackage{amssymb}
\usepackage{amsthm}
\usepackage{amssymb, amsmath}
\usepackage{enumerate}

\newcommand{\CJS}{CJS}

\usepackage{tikz}
\usetikzlibrary{arrows.meta,positioning,shapes}

\usepackage{marginnote}
\usepackage{enumerate}  

\usepackage[normalem]{ulem}
\usepackage{url}

\newtheorem{definition}{Definition}[section]
\newtheorem{dfn}[definition]{Definition}
\newtheorem{exa}[definition]{Example}
\newtheorem{rem}[definition]{Remark}
\newtheorem{prop}[definition]{Proposition}
\newtheorem{lem}[definition]{Lemma}
\newtheorem{coro}[definition]{Corollary}
\newtheorem{teo}{Theorem}

\newtheorem*{cla*}{Claim}

\newtheorem*{teo_green*}{Green's function Theorem}
\newtheorem*{teo*}{Theorem}

\newcommand{\N}{I\!\!N}

\usepackage[pagewise,displaymath,mathlines]{lineno}
\begin{document}
	
%	\linenumbers
	
	\bigskip

	\title[Random Dynamicss]{Random Dynamics of a Family of Cubic Polynomials}
	
	\author[]{A. M. Alves}
	\address[Alexandre Miranda Alves]{Departamento de Matem\'atica, Universidade Federal de Vi\c cosa, Brazil}
	\email{amalves@ufv.br }
	
	\author[]{G. Honorato}
	\address[Gerardo Andrés Honorato Gutiérrez]{Instituto de Ingeniería Matemática-CIMFAV, Universidad de Valparaíso, Chile}	
	\email{gerardo.honorato@uv.cl}
	
	%	\author[]{P. Mehdipour}
	%	\address[Pouya Mehdipour]{Departamento de Matem\'atica, Universidade Federal de Vi\c cosa, Brazil
		%	}
	%	\email{???}
	%	
	%	\author[]{D. S. Machado}
	%	\address[Diogo Machado]{Departamento de Matem\'atica, Universidade Federal de Vi\c cosa, Brazil}
	%	\email{diogo.machado@ufv.br }
	
	\author[]{M. Salarinoghabi}
	\address[Mostafa Salarinoghabi]{%
		Departamento de Matem\'atica, Universidade Federal de Vi\c cosa, Brazil.
	}
	\email{mostafa.salarinoghabi@ufv.br}
	


	\begin{abstract}
In this work, we study the non-autonomous dynamics generated by random iterations of the cubic family of the form $z^3 + cz$. The parameter sequence is chosen randomly from a bounded Borel subset of $\mathbb{C}$. We investigate topological properties of the corresponding Julia sets, with particular emphasis on conditions leading to total disconnectedness. We prove that the set of parameter sequences for which the Julia set is totally disconnected is dense in the parameter space. We also construct examples where the Julia set is totally disconnected but the associated non-autonomous system is not hyperbolic. Finally, under suitable probabilistic assumptions on the parameter distribution, we show that almost every sequence produces a totally disconnected Julia set.
	
%%In nature and in many mathematical models, the parameters of a system are not fixed but may vary randomly. Random complex dynamics provides a natural framework for studying how randomness influences the global behaviour of orbits and the structure of Julia sets. In this paper 	 we aim to investigate the topological and dynamical properties of a family of random cubic polynomials. Specifically, in this paper we study some conditions for which the Julia set of family of random cubic polynomials $z^3+c_n z$ is totally disconnected where the sequence $(c_n)$ is chosen randomly from a bounded Borel set in $\mathbb{C}$.
	\end{abstract}	
	
	\maketitle	
	
	%%%%%%%%%%%%%%%%%%%%%%%%%%%%%%%%%%%
	%\tableofcontents
	%%%%%%%%%%%%%%%%%%%%%%%%%
	\section{Introduction}%\label{sec:introduction}
The concept of Julia-Fatou sets plays a fundamental role in the study of complex dynamics, particularly in understanding the intricate dynamical behavior of polynomials and rational functions. There are several natural phenomena whose mathematical models can be locally conjugated to complex polynomials. In such models there is always some noise, for example, nonlinear wave propagation in dispersive media or certain security-related aspects of 5G systems, especially in the study of complex dynamical systems and noise and disturbance analysis.

When examining the dynamics of cubic polynomials, the Julia set can exhibit a diverse range of structures, from connected fractal curves to totally disconnected patterns. In recent years, there has been growing interest in the study of Julia sets of random polynomials, where the coefficients are randomly selected from a predefined or controlled set. This probabilistic approach unveils new dimensions in the behavior of these sets, introducing randomness into their formation and structure.

Specifically, for families of random cubic polynomials, the investigation into their Julia sets provides fascinating insights into the interplay between randomness and dynamical properties.

A particularly intriguing phenomenon observed in Julia sets of random cubic polynomials is total disconnectedness. This means that the Julia set consists of an uncountable number of points, sometimes resembling a Cantor set. Studying totally disconnected Julia sets not only deepens our understanding of fractal geometry, but also sheds light on the complex nature of dynamical systems under the influence of randomness.

The study of iterations of random non-autonomous rational maps originated in the seminal paper \cite{FS} by Fornæss and Sibony. Since then, many mathematicians have worked on this subject (see for instance \cite{BB,B1,Co1,Co2,FS,chines,chines2,totall}).  A systematic study of the dynamical properties of random non-autonomous dynamical systems of quadratic polynomials was done by Brück, Büger and Reitz; see \cite{Br,Br2,B1}. The behavior of critical points of a random family of polynomials plays an important role in the study of non-autonomous family of polynomials. 

Because of many interesting applications that a certain cubic family of polynomials have, we decide to study some topological and dynamical properties of this family in a series of papers. To be more precise, in these series of papers we consider the family of cubic polynomials,
\begin{equation}\label{eq1}
f_{\omega}(z)=z^3+c_nz,
\end{equation}
where the sequence $\omega=(c_n)$ is chosen randomly from some bounded Borel subset $\Omega$, where $\Omega = W^{\mathbb{N}}$ and $W$ is a bounded Borel subset of $\mathbb{C}$. Throughout this work, we will often consider $\Omega = D_{\delta}^{\mathbb{N}}$, where $D_{\delta}$ denotes the open disk in $\mathbb{C}$ of radius $\delta > 0$.

In discrete dynamical systems, the iterations of a point reveal many outstanding information about the dynamic (i.e, our model) itself.  
In random systems, for a given sequence of complex parameters $\omega=(c_n)$ the iteration of a point is defined by the composition 
\begin{equation}\label{eq:Fn}
f^n_{\omega}(z) = f_{c_n}\circ \dots \circ f_{c_1}(z).
\end{equation}
It is worth mentioning that if the sequence $(c_n)$ is constant, i.e., $c_n = c$ for all $n\in\mathbb{N}$, then the above iteration is basically the classic well-known iteration of the polynomial $f_c$. We devote \S\ref{sec:classic} to some preliminary results on the classic dynamic of cubic polynomials.  

For a sequence of complex parameters $\omega=(c_n)\in\Omega$, we consider the non-autonomous composition of functions $f_{c_n}$ given in \eqref{eq1} and its $n^{\text{th}}$ iterate of a point $z\in\widehat{\mathbb{C}}=\mathbb{C}\cup\{\infty\}$, as given in \eqref{eq:Fn}.  Furthermore, set $f_{\omega}=f_{c_1}$. 

The definitions of Julia set (resp. Fatou set), denoted by $\mathcal{J}_{\omega}$ (resp. $\mathcal{F}_{\omega}$) are naturally generalized to this setting. To be more precise, the Fatou set is the set of all $z\in \widehat{\mathbb{C}}$ such that the family $\left(f_{\omega}^n\right)$ is normal (in the sense of Montel) in some neighborhood of $z$. The complement of the Fatou set is called the Julia set. A component of $\mathcal{F}_{\omega}$ is called a stable domain, or a Fatou component. The filled Julia set $\mathcal{K}_{\omega}$ is by definition the set of all $z\in \mathbb{C}$ whose orbit of $\left(f_{\omega}^n\right)$ is bounded. Clearly, the set $\mathcal{J}_{\omega}$ is compact and  $\mathcal{F}_{\omega}$ is open, both in $\widehat{\mathbb{C}}$. Additionally, we have $\partial\mathcal{K}_{\omega} =\mathcal{J}_{\omega}$, see \cite{BB,FS} for instanse. Also $\mathcal{A}_{\omega}(\infty) = \widehat{\mathbb{C}} \setminus \mathcal{K}_{\omega}$ is the basin of infinity corresponded to the sequence $\omega$. A domain $M_{\omega}$ is said to be invariant with respect to $f_{\omega}$, if 
$f_{c_k}(M_{\omega})\subset M_{\omega} \mbox{ for all } k\in\mathbb{N}$.

Our focus in this paper is the study of connectedness (in particular totally disconnectedness) of the Julia set $\mathcal{J}_{\omega}$ of the family of non-autonomous cubic polynomials \eqref{eq1}. This subject has been widely investigated for a random family of quadratic polynomials $z^2+c_n$ where $(c_n) \in\Omega$, see for instance \cite{BH,Br,Br2,BB,B1,CCV,chines,chines2,totall}. The lack of a deep study for the family of random cubic polynomials given in \eqref{eq1} motivated us to focus on this topic. 

It is worth emphasizing that the behavior of the critical points of a classical or random family of polynomials plays a fundamental role in the study of autonomous and non-autonomous families of polynomials (for the case of autonomous dynamics see \S\ref{sec:classic}). 

The connectedness and totally disconnectedness of the Julia set of a  special random family of polynomials, which are called {\it bounded hyperbolic}, is well studied by many authors (for definitions and more details see \cite{Co1,Co2} for example). 

The non-autonomous setting is too broad, allowing pathological behaviors of Julia sets. Thus, Comerford restricts attention to sequences $\{f_m\}_{m=1}^{\infty}$ of polynomials 
with uniformly bounded degrees ($2 \le d_m \le d$) and uniformly bounded coefficients, see \cite{Co1}. Such sequences are called \emph{bounded sequences of polynomials}. Moreover, Comerford  generalized the definition of hyperbolicity to the non-autonomous setting, as follows.

\begin{definition}\label{hyperbolic}
Let $\omega = (f_n)_{n \ge 1}$ be a sequence of polynomials on $\mathbb{C}$ and denote by 
\[
F_{m,n} = f_{m+n} \circ f_{m+n-1} \circ \cdots \circ f_{m+1}, \qquad F_n = F_{0,n}.
\]
We say that the sequence $\omega$ (or the corresponding non-autonomous dynamical system) is 
\emph{hyperbolic} if there exist constants $\lambda > 1$ and $C > 0$ such that
\[
\big|(F_{m,n})'(z)\big| \;\ge\; C\,\lambda^n
\quad\text{for all } z \in J_m,\; \text{and all } m,n \ge 0,
\]
where $J_m$ denotes the Julia set at step $m$.
\end{definition}
In other words, $\omega$ is hyperbolic if the derivative growth is uniformly expanding on the sequence of Julia sets $\{J_m\}$. In \cite{Co2}, Comeford proves the following important results:

\begin{teo*}
	If all the critical points of a bounded hyperbolic polynomial sequence $\{F_n\}_{n=1}^{\infty}$, where $F_n(z) = f_{c_n}\circ\dots\circ f_{c_1}$, escape to infinity under iteration,
	\begin{itemize}
		\item[(1)]  then there are no bounded Fatou components (\cite[Theorem 1.2]{Co2});
		\item[(2)]  then all iterated Julia sets are totally disconnected (\cite[Theorem 1.3 ]{Co2}).
	\end{itemize}
\end{teo*}

In the Example \ref{exe1}, we exhibit a bounded sequence of polynomials from family \eqref{eq1} whose criticals points escape to infinity, but whose filled Julia set contains an interval. It followed that the Julia set for such a sequence is not totally disconnected, so in view of Theorem 1.3 de \cite{Co2}, such a sequence cannot be hyperbolic.

\begin{exa}\label{exe1}
	We construct a bounded sequence $\omega \in \mathbb{C}$ whose corresponding 
	Julia set is disconnected but not totally disconnected. 
	Set $c_1 = -2$ and $c_n = 0$ for all positive integers $n \geq 2$. 
	The critical points of $f_{c_1}$ are $$z_{+} = +\sqrt{\tfrac{2}{3}}, \qquad z_{-} = -\sqrt{\tfrac{2}{3}}.$$ 
    
    In this case, $f_{c_1}(z_-) \cong 1.088>1$. Since the following steps of the map are given by $z \mapsto z^{3}$, we have that $$|f^n_{\omega}(z_-)| \to \infty,$$ 
	as $n \to \infty$. Hence, by Theorem \CJS~ (see page 6), the Julia set is disconnected. 
	Furthermore, we have $f_{c_1}\left((0, \tfrac{1}{2})\right) \subset (-1, 0)$ and 
	$f^n_{\omega}\left((0, \tfrac{1}{2})\right) \subset (-1, 0)$ for 
	all $n \in \mathbb{N}$. Indeed, on the interval $(0,\tfrac{1}{2})$ the function 
\[
x \mapsto x^{3} - 2x
\]
is strictly decreasing; it maps $0$ to $0$ and $\tfrac{1}{2}$ to $-\tfrac{7}{8}$. So, 
\[
f_{c_{1}}\bigl((0,\tfrac{1}{2})\bigr) = \bigl(-\tfrac{7}{8},0\bigr) \subset (-1,0).
\]
For $y \in (-1,0)$ we have $y^{3} \in (-1,0)$ and $|y^{3}| < |y|$; therefore, under forward iterations $y \mapsto y^{3k} \to 0$. We conclude that 
\[
(0,\tfrac{1}{2}) \subset K_{\omega}.
\]
Since $K_{\omega}$ contains a non-degenerate real interval, its boundary $J_{\omega}$ cannot be totally disconnected. We have shown that $J_{\omega}$ is disconnected and not totally disconnected. This phenomenon was already observed in the quadratic case by Brück \cite{ Br} and Brück–Büger–Reitz \cite{Br2}.
\end{exa}

\section{An example totally disconnected and non hyperbolic}

We present an example of a sequence of cubic polynomials that is totally disconnected but not hyperbolic in the sense of Definition \ref{hyperbolic}.
The idea of the proof is as follows. 
We consider highly expanding blocks, which force a Cantor-like structure with isolated steps in which the derivative is close to one. 
The long blocks ensure that the Julia set is totally disconnected, while the infinite presence of these neutral times prevents the uniform expansion required by hyperbolicity. 
This results in a totally disconnected but non-hyperbolic Julia set within of the family \eqref{eq1}. 
This mechanism resembles the Pliss lemma from the theory of non-uniformly hyperbolic diffeomorphisms.

While the classical lemma guarantees the existence of infinitely many hyperbolic times under average expansion, our construction considers infinitely many near-parabolic steps within arbitrarily long expanding blocks. 
As a result, the system remains totally disconnected but breaks the hyperbolicity. See, for instance, \cite{Barreira-Pesin}.

\begin{definition}[Bounded non-autonomous system]\label{def:bounded}
Let $S=(f_n)_{n\ge1}$ be a sequence of rational maps on $\mathbb{\widehat{C}}$. 
We say that $S$ is \emph{bounded} if there exists $d_\ast\in\N$ and a compact set 
$\mathcal{F}\subset \mathrm{Rat}_{\le d_\ast}$ (in the usual topology of coefficients, after a fixed normalization) such that 
$f_n\in\mathcal{F}$ for all $n$. 
Equivalently, the degrees are uniformly bounded, $2\le \deg(f_n)\le d_\ast$, and the normalized coefficients of $f_n$ remain in a compact set (no degeneration to constants).
\end{definition}

\begin{rem}
For polynomials, one may normalize (e.g. monic and centered, or any fixed affine normalization) 
and require that all coefficients lie in a compact set; this is equivalent to boundedness above. 
In particular, boundedness gives uniform distortion, compactness over domains used in Markov blocks.
\end{rem}

\begin{exa}[Cubic family]\label{exa:bounded-cubic}
In the family $f_{\omega}(z)=z^3+c_nz$, a sequence $(f_{c_n})$ is bounded if and only if the parameter sequence $(c_n)$ is bounded in $\mathbb{C}$, i.e.\ $\sup_n |c_n|<\infty$. 
\end{exa}

\begin{definition}[Near-parabolic step in non-autonomous dynamics]\label{def:nearparabolic-general}
Let $S=(f_n)_{n\ge1}$ be a bounded sequence of rational maps on $\mathbb{\widehat{C}}$, and write $F_{m,n}=f_{m+n}\circ\cdots\circ f_{m+1}$ and $J_m$ for the stage-$m$ Julia set.
We say that an index $m\in\N$ is a \emph{near-parabolic step} if there exists $z_m\in J_m$ with
\[
\big|f'_m(z_m)\big| \;=\; 1+\varepsilon_m
\quad\text{for some } \varepsilon_m>0 \text{ small}.
\]
We say that $\omega$ has \emph{infinitely many near-parabolic steps} if there exists an infinite set 
$\{m_k\}_{k\ge1}$ with $\varepsilon_{m_k}\to0$.
\end{definition}

The terminology ``near-parabolic'' is standard in complex dynamics to describe maps whose local behavior approaches the parabolic regime $|f'(z)|=1$ without attaining its value. Here it singles out non-expanding instants along the sequence that are arbitrarily close to neutrality on $J_m$.

\begin{prop}[Near-parabolic steps destroy uniform hyperbolicity]\label{prop:np-not-hyp}
Let $S=(f_n)$ be a bounded non-autonomous system as above. 
If $S$ contains infinitely many near-parabolic steps, then $\omega$ is not hyperbolic.
\end{prop}

\begin{proof}
Fix any $\lambda>1$ and $C>0$. Choose $k$ large so that $|f'_{m_k}(z_{m_k})|=1+\varepsilon_{m_k}<\lambda$
with $z_{m_k}\in J_{m_k}$ from Definition~\ref{def:nearparabolic-general}. Then
$\big|(F_{m_k,1})'(z_{m_k})\big|=|f'_{m_k}(z_{m_k})|<\lambda$, contradicting any lower bound
$C\lambda$ at time $n=1$. Since such $m_k$ occur infinitely often with $\varepsilon_{m_k}\to0$,
no uniform expansion factor $\lambda>1$ can hold along the tail. Hence the system is not hyperbolic.
\end{proof}

\begin{prop}[Tail-imposed Cantor structure via Markov blocks]\label{prop:tail-cantor}
Assume there exist infinitely many \emph{Markov blocks} indexed by $k$, that is, intervals 
$[N_k,N_k+L_k)$ with $L_k\to\infty$ such that for each $n$ in the block, $f_n:U\to V$ is a Markov map (in the sense of a topological disk $V$, a finite disjoint union $U=\bigsqcup_j U_j\Subset V$ (here, $\Subset$  means compactly contained), and each restriction $f_n:U_j\to V$ univalent onto). Then for each $k$ there exists a compact $C_k\subset J_{N_k}$ homeomorphic to a Cantor set, and the diameters of level-$1$ cylinders in $C_k$ tend to $0$ as $k\to\infty$.
Consequently, the non-autonomous Julia set $J_\omega$ is totally disconnected (a Cantor set).
\end{prop}

\begin{proof} By the Markov-map theorem (see, for instance \cite{deFaria}), each block induces an IFS of $3^{L_k}$ inverse branches on $V$ that is uniformly contractive in the hyperbolic metric. Its attractor is a Cantor set $K_k\subset U$, 
topologically conjugate to a (sub)shift. Transporting $K_k$ back to the start of the block through the finitely many inter-block steps gives a compact $C_k\subset J_{N_k}$ with the same clopen partition and exponentially small mesh. As $L_k\to\infty$, meshes $\to0$ arbitrarily deep in the tail, yielding total disconnectedness.
\end{proof}
\begin{figure}[ht]
\centering
\begin{tikzpicture}[scale=0.8,>=Stealth,thick]

% ==== Time axis (upper part) ====
\draw[->,thick] (0,0) -- (13,0) node[below right] {$n$ (time)};
\draw[->,thick] (0,0) -- (0,3.0) node[left] {$|f'_n|$};

% Hyperbolic (Markov) blocks
\fill[blue!15] (0.3,0) rectangle (2.5,2.5);
\fill[blue!15] (3.5,0) rectangle (7.0,2.5);
\fill[blue!15] (8.0,0) rectangle (12.3,2.5);

\draw[blue,thick] (0.3,2.5)--(2.5,2.5);
\draw[blue,thick] (3.5,2.5)--(7.0,2.5);
\draw[blue,thick] (8.0,2.5)--(12.3,2.5);

\node[blue!70!black] at (1.4,2.8) {\scriptsize Block 1};
\node[blue!70!black] at (5.2,2.8) {\scriptsize Block 2};
\node[blue!70!black] at (10.0,2.8) {\scriptsize Block 3};

% Near-parabolic steps
\foreach \x in {3.0,7.5,12.6}{
  \draw[red!70!black,thick] (\x,0)--(\x,1.0);
  \fill[red!70!black] (\x,1.0) circle(0.06);
}
\node[red!70!black] at (3.0,1.3) {\scriptsize near-parabolic};
\node[red!70!black] at (7.5,1.3) {\scriptsize near-parabolic};
\node[red!70!black] at (12.6,1.3) {\scriptsize near-parabolic};

% Dashed line at |f'|=1
\draw[dashed,gray!70] (0,1.0)--(13,1.0);
\node[gray!70] at (0.7,1.15) {\scriptsize $|f'_n|=1$};

% ==== Cantor refinement (lower part) ====
\begin{scope}[yshift=-3.5cm]
  % Horizontal base line (Julia set schematic)
  \draw[thick] (0,0)--(13,0);
  \node[below right] at (13,0) {\small phase space};

  % Block 1 refinement
  \foreach \x in {0.5,1.1,1.7,2.3}{
    \draw[blue!70!black,thick] (\x,-0.05)--(\x,0.6);
  }
  \node[blue!70!black] at (1.4,0.8) {\scriptsize 3 pieces};

  % Block 2 refinement (finer)
  \foreach \x in {3.9,4.2,4.5,4.8,5.1,5.4,5.7,6.0,6.3,6.6,6.9}{
    \draw[blue!50!black,thick] (\x,-0.05)--(\x,0.8);
  }
  \node[blue!70!black] at (5.3,1.0) {\scriptsize $3^{L_2}$ subpieces};

  % Block 3 refinement (very fine)
  \foreach \x in {8.2,8.3,8.4,...,12.0}{
    \draw[blue!30!black,thick] (\x,-0.05)--(\x,0.9);
  }
  \node[blue!70!black] at (10.2,1.1) {\scriptsize finer clopen partition};

  % Labels
  \node at (1.4,-0.5) {\scriptsize after Block 1};
  \node at (5.3,-0.5) {\scriptsize after Block 2};
  \node at (10.2,-0.5) {\scriptsize after Block 3};
\end{scope}

% Arrows connecting time to phase space
\foreach \x/\y in {1.4/-3.5,5.3/-3.5,10.2/-3.5}{
  \draw[->,gray!70,thick] (\x,0)--(\x,\y+2.5);
}

\end{tikzpicture}

\caption{Top: long hyperbolic (Markov) blocks (blue) enforce expansion, 
interspersed with isolated near-parabolic steps (red) where $|f'_n|\approx1$. 
Bottom: each block induces a finer clopen partition of the Cantor-like Julia set, 
showing that total disconnectedness persists arbitrarily deep in the tail.}\label{fig:time-cantor}
\end{figure}
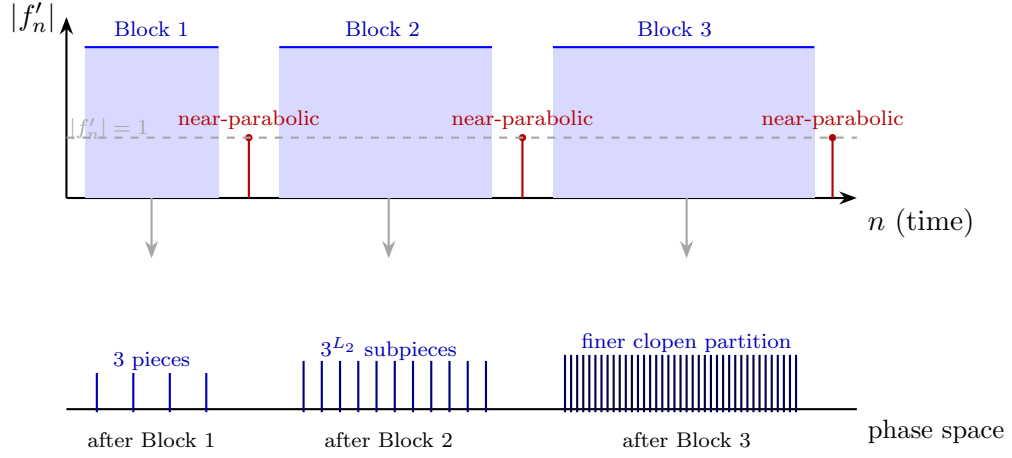

\begin{coro}[Non-hyperbolic Cantor Julia set ]\label{cor:cantor-nonhyp}
If, in addition to the hypotheses of Proposition~\ref{prop:tail-cantor}, the sequence $S$ contains infinitely many near-parabolic steps, then $J_\omega$ is totally disconnected while $S$ is not hyperbolic.
\end{coro}

\begin{proof}
\label{exa:cubic-nearpara}
Consider the family $f_c(z)=z^3+c\,z$ and a bounded sequence $\omega=(c_n)_{n\ge1}$ built as follows.
Fix $C^*$ with $|C^*|$ large so that $f_{C^*}:U\to V$ is Markov (three univalent branches onto $V$).
Let $[N_k,N_k+L_k)$ be blocks with $c_n\equiv C^*$ and $L_k\to\infty$, and let $m_k:=N_k+L_k$ be single
intervening steps. For each $k$, choose a repelling point $r_k$ for the next block and a deep inverse image
$z_k$ under $(f_{C^*})^{\circ L_k}$ so that $|3z_k^2+c_{m_k}|=1+\varepsilon_k\to1$, defined by
\[
c_{m_k}:=\alpha_k-3z_k^2,\qquad \alpha_k:=\frac{r_k+2z_k^3}{z_k},\quad |\alpha_k|=1+\varepsilon_k.
\]
Then $f_{c_{m_k}}(z_k)=r_k$ with $\big|f'_{c_{m_k}}(z_k)\big|=1+\varepsilon_k$, hence $m_k$ are near-parabolic steps.
By Proposition~\ref{prop:tail-cantor} the Markov blocks enforce a Cantor tail, and by
Proposition~\ref{prop:np-not-hyp} the system is not hyperbolic. Therefore $J_\omega$ is totally disconnected
but $S$ is non-hyperbolic.

\end{proof}

\section{Topological properties of $\mathcal{J}_{\omega}$ and main results}

As in the classical case, in the non-autonomous setting, critical points play a fundamental role in understanding the dynamical and topological properties of the Julia set. In the non-autonomous setting, the critical set is defined by, 
	\[\mathcal{C}_{\left(f^n_{\omega}\right)}=\left\{ z \in \mathbb{C}; (f^{m}_{\omega})^{\prime}(z)=0, \mbox{ for some } m \in \mathbb{N}\right\}.\] 
Since
\[
(f^{m}_{\omega})^{\prime}(z)=\prod_{j=0}^{m-1} f^{\prime}_{c_{m-j}}\left(f^{m-j-1}_{\omega}(z)\right),
\]
this set may also be described as follows: if $\mathcal{C}_{j-1}=\left\{z \in \mathbb{C}; f_{c_{j}}^{\prime}(z)=0\right\}$, then $\mathcal{C}_{\left(f^n_{\omega}\right)}$ is the set of $z \in \mathbb{C}$ such that $f_{\omega}^{j-1}(z) \in \mathcal{C}_{j-1}$ for some $j \in \mathbb{N}$. From the existence of Green function (see \S\ref{sec:green}) it is possible to deduce a necessary and sufficient condition for the Julia set to be connected.

The Theorem \CJS , see, for exemple, \cite{Br,BB,FS}, presents a necessary and sufficient condition for $\mathcal{J}_{\omega}$ to be connected. 
\begin{teo*}[CJS]\label{CJS}
The Julia set $\mathcal{J}_{\omega}$ is connected if and only if $\mathcal{C}_{\left(f^n_{\omega}\right)}\subset \mathcal{K}_{\omega}$.
\end{teo*}

In a similar way to what was done in \cite{Br,BB}, we introduce subsets of $\Omega$ described in terms of connectedness of Julia set.
	\begin{eqnarray}
	\mathcal{D} &=&\{\omega = (c_n)\in \Omega ;~\mathcal{J}_{\omega} ~\text{is disconnected}\}, \label{eq:D}\nonumber \\
	\mathcal{D}_N &=&\{\omega = (c_n)\in \Omega ;~\mathcal{J}_{\omega} ~\text{has more than } N~\text{components} \}, \label{eq:DN} \nonumber\\
	\mathcal{D}_{\infty} &=&\{\omega = (c_n)\in \Omega ;~\mathcal{J}_{\omega} ~\text{has infinitely many components} \}, \label{eq:Dinf} \nonumber\\
	\mathcal{T} &=&\{\omega = (c_n)\in \Omega ;~\mathcal{J}_{\omega} ~\text{is totally disconnected } \}, \label{eq:T} \nonumber\\
	\mathfrak{F} &=&\{\omega = (c_n)\in \Omega ;~ \forall \hat{z}\in \mathcal{C}_{(f^n_{\omega})},\,\, \hat{z}\text{ is  fast escaping}  \}. \label{eq:infinity} \label{eq:FF}
\end{eqnarray} 
See Definition \ref{def:typically fast} for fast escaping.

Clearly, for $N>1$, we have
\[
\mathcal{D}  \supset \mathcal{D}_N  \supset \mathcal{D}_{\infty}  \supset \mathcal{T},
\]
but the set $\mathfrak{F}$ is neither contained in nor contains $\mathcal{T}$ (later, in \S\ref{sec:Top} we deal with this question). We shall recall that, for the random family of quadratic polynomials happens a similar situation with the difference that in that case, the set $\mathfrak{F}$ is the set of sequences of parameters $\omega$ such that $f^n_{\omega}(0)\to\infty$, due to the fact that zero is the only critical point of $f_{\omega}^n$.

The remainder of this paper is organized as follows. In the rest of this section, we state the main results of this work. 
In section \S\ref{sec:classic}, we present the classical dynamical setting concerning the topological properties of the cubic family \eqref{eq1}, where $c_n$ is constant. In Section \S\ref{sec:green}, we present the Green function for the scenario of the non-autonomous dynamics. In section \S\ref{sec:Top} we present topological properties of $\mathcal{J}_{\omega}$ and the proofs of the main results. The main results of this work are described below.

\begin{teo}\label{teo: T dense}
	Suppose that $\omega\in\Omega$ with $\delta > 3$. The set $\mathcal{T}$ defined by \eqref{eq:T} is dense in $\Omega$.
\end{teo} 
\begin{teo}\label{teo:D}
	The set $\mathcal{D}$ given in \eqref{eq:FF} is an open and dense subset of $\Omega$ provided that $\delta>3$.
\end{teo}

For the rest of this section we concentrate on finding conditions to have totally disconnected Julia set.

Let $R>R_{\delta}:=\sqrt{1+\delta}$, and $\tilde{R} \in (R, R^3-\delta R)$. 
Choose an arbitrary radius $\rho\in [R,\tilde{R}]$ and $k\in\mathbb{N}$, put  
\[
B_k = (f^k_{\omega})^{-1}(D_{\rho}), 
\]
which is a union of pairwise disjoint topological discs $B_{k,j}$. These topological discs are being mapped by a $d_j^k$-degree map $f^k_{\omega}$ to $D_{\rho}$, where $d_j < 3$ is a non-zero real number.  In the following theorem we deal with the case $\rho = \tilde{R}$ (see Lemma \ref{lem:R tilde} for definitions of these notations), and represent a sufficient condition for total disconnectedness of the Julia set $\mathcal{J}_{\omega}$.
% It is worth mentioning that, a same result is valid for the quadratic random family of polynomials $z^2+c_n$ (see \cite[Proposition 10]{totall}).

\begin{definition}\label{def:typically fast}
	Let $W\subset\mathbb{C}$ be a bounded Borel set with $W\subset D_{\delta}$ for some $\delta > 0$. Also let $\mu_{\delta}$ be a probability Borel measure on $W$, and $\mathbb{P}={\bigotimes}_{n=0}^{\infty}\mu_{\delta}$ the product distribution on $W^{\mathbb{N}}$ generated by $\mu_{\delta}$. Suppose that $G_{\delta}$ are as given in Lemma \ref{lem:R tilde}. For $\omega\in W^{\mathbb{N}}$, let $z\in\mathbb{C}$, we say that $z$ is typically fast escaping if there exists $\gamma>0$ such that 
	\[
	\mathbb{P}\left(A_k(z)\right) < e^{-\gamma k},
	\]   
	where 
	\[
	A_k(z) = \{\omega\in\Omega ~;~g_{\omega}(z) < \frac{G_{\delta}}{3^k}\}.
	\]
\end{definition}
%\begin{teo}\label{teo:at most N}
%	Let $\omega \in D_{\delta}^{\mathbb{N}}$, with $\delta > 0$. If there exists a natural number $N$ such that for finitely many $k\in\mathbb{N}$ and for each component $B_{k,j}$ of $B_k$, the degree of the map $f_{\omega}^k : B_{k,j} \to D_{\tilde{R}}$ is at most $N$. Then the Julia set $\mathcal{J}_{\omega}$ is totally disconnected.
%\end{teo}
\begin{teo}\label{teo:at most N}
	Let $\omega \in D_{\delta}^{\mathbb{N}}$, with $\delta > 0$. If there exists a natural number $N$ such that for finitely many $k\in\mathbb{N}$ and for each component $B_{k,j}$ of $B_k$, the degree of the map $f_{\omega}^k : B_{k,j} \to D_{\tilde{R}}$ is at most $N$. Then the Julia set $\mathcal{J}_{\omega}$ is totally disconnected (i.e. $\omega\in\mathcal{T}$).
\end{teo}
\begin{teo}\label{teo:Ak}
	Let $\hat{z}\in\mathcal{C}_{(f^n_{\omega})}$ be an arbitrary critical point of $f^n_{\omega}$. 
%	Set
%	\[
%	A_k := \bigcap_{\hat{z}\in\mathcal{C}_{(f^n_{\omega})}} A_k(\hat{z}). 
%	\]
	\begin{itemize}
		\item[(a)] For all $i=0,\dots,k-1$, if $\sigma^i\omega\notin A_{k-i}(\hat{z})$ then, for every connected component $B_{k,j}$ of $B_k$, the map $f^k_{\omega} : B_{k,j} \to D_{\tilde{R}}$ has degree 1 (Here, $\sigma$ denotes the shift function). 
		\item[(b)] If this fails for $0\leq l\leq k-1$ indices (i.e., $\sigma^i\omega \notin A_{k-i}(\hat{z})$ for $l$ indices) then the degree of the map $f^k_{\omega}$ is at most $N=3^l$.  
	\end{itemize}
Therefore, in both cases the Julia set $\mathcal{J}_{\omega}$ is totally disconnected ($i.e., \omega\in\mathcal{T}$). 
\end{teo}

\begin{teo}\label{teo: P almost every}
	Let $W$, $\mu_{\delta}$ and $\mathbb{P}$ be as given in Definition \ref{def:typically fast}. If every critical point of $f^n_{\omega}$ is fast scaping, then the assumptions of Theorem \ref{teo:at most N} are satisfied for $\mathbb{P}$ almost every $\omega\in W^{\mathbb{N}}$. Therefore, for $\mathbb{P}$ almost every $\omega\in W^{\mathbb{N}}$ 
	the Julia set $\mathcal{J}_{\omega}$ is totally disconnected.
\end{teo}
We set the following notation and definition: for every natural numbers $k>j$ and $\omega\in\Omega$, set $
f^{k,j}_{\omega}(z) = f_{c_k}\circ \dots \circ f_{c_{j+1}}(z).$ 
Define
\begin{eqnarray*}
	\mathfrak{G} = \big\{z\in\widehat{\mathbb{C}}&;&\,\exists \text{ a bounded simply connected domain }\mathfrak{D} \subset \mathbb{C} \\
	&& \text{ such that } \overline{\bigcup f_{\omega}^k(\mathcal{J}_{\omega})} \subset \mathfrak{D} \text{ and } f^{k,j}_{\omega}(z) \notin \mathfrak{D}\\
	&& \text{ for all } j=0,\dots, k-1 \text{ and } k\in\mathbb{N} 
	\big\}.
\end{eqnarray*}
%\begin{teo}\label{teo7}
%	Let $\omega=(c_n)\in (\overline{D}_{\delta})^{\mathbb{N}}$ for $\delta>3$ be a sequence on the axes of coordinates. If $\mathcal{C}_{(f^n_{\omega})} \subset \mathfrak{G}$ then the Julia set $\mathcal{J}_{\omega}$ is totally disconnected.
%\end{teo}
\begin{teo}\label{teo7}
	Let $\omega=(c_n)\in (\overline{D}_{\delta})^{\mathbb{N}}$ for $\delta>3$ be a sequence on the coordinate axes. If $\mathcal{C}_{(f^n_{\omega})} \subset \mathfrak{G}$ then $\omega\in\mathcal{T}$
\end{teo}

%%%%%%%%%%%%%%%%%%%%%%
%%%%%%%%%%%%%%%%%%%%%%
%\newpage
\section{The scenario in classical dynamics} \label{sec:classic}

Let $f$ be a complex polynomial of degrees $d\geq 2$ defined on the Riemann sphere $\widehat{\mathbb{C}}= \mathbb{C}\cup\{\infty\}$. The term $f^n=f\circ f^{n-1}$ denotes the $n$-th iterate of $f$. Under the iterations of $f$ the Riemann sphere decompose into two sets, the Fatou set $\mathcal{F}_f$, which is the set where the family $\{f^n\}$ is normal in the sense of Montel, and the complement of Fatou set which is called the Julia set and is denoted by $\mathcal{J}_f$. The set
\[\mathcal{K}_f=\{z\in \mathbb{C}; \{f^n(z)\mbox{ is bounded}\}\},\]
is called the filled-in Julia set of $f$. The Julia set $\mathcal{J}_f$ of $f$ is the boundary of $ \mathcal{K}_f$, see \cite[p. 95]{Milnor}.

We say that $z$ is a finite critical point of $f$ if $f'(z)=0$. The point $z=\infty$ is a critical point of every polynomial $f$.

We say that a set $K$ is totally disconnected if each component of $K$ consists of a single point. The following proposition is a topological result, but it depends on the properties of $\widehat{\mathbb{C}}$.
\begin{prop}\label{pro1}
	If $K\subset\widehat{\mathbb{C}}$ is a totally disconnected set then the complement of $K$ is connected.
\end{prop}
\begin{proof}
	In \cite[Theorem 14.3]{MN}, we have the following result,  if $x,y\in \widehat{\mathbb{C}}$ are separated by the closed set $K$ then they are separated by a component of $K$. If $\widehat{\mathbb{C}}\setminus K$ is not connected then there are points $x, y\in \widehat{\mathbb{C}}\setminus K$ separated by $K$. In this case, $x$ and $y$ are separated by a component of $K$. But this is a contradiction because each component of K is a single point.
\end{proof}

A component of $\mathcal{K}_f$ is called critical if it contains critical points.
We denote the component of $\mathcal{K}_f$ containing $z$ by $\mathcal{K}_f(z)$. A component $\mathcal{K}_f(z)$ of  $\mathcal{K}_f$ is called aperiodic if
$f^n(\mathcal{K}_f(z))\neq \mathcal{K}_f(z)$ for all $n > 0$.

About the connectivity of the Julia set, Fatou and Julia proved the following theorem, see \cite{Fatou,Julia}.

\begin{teo*}[Julia-Fatou]
	\begin{enumerate}
		\item The Julia set of a complex polynomial $f$ is connected if and only if $\mathcal{K}_f$ contains all critical points of $f$.
		\item The Julia set of a complex polynomial $f$ is a Cantor set if $\mathcal{K}_f$ contains no critical points of $f$.
	\end{enumerate}
\end{teo*}

Fatou conjectured (see \cite[pag 84]{Fatou}) that if a critical point belongs to $\mathcal{K}_f$, then $\mathcal{K}_f$ cannot be totally disconnected. This was disproved by Brolin, see \cite[Theorem 13.8]{Brolin}. Brolin presented examples of real cubic polynomials, where $\mathcal{J}_f$ is a Cantor set, with $\mathcal{J}_f =\mathcal{K}_f$, and $\mathcal{K}_f$ containing a critical point.

Branner and Hubbard completely settled the question of when the Julia set of a cubic polynomial is a Cantor set.  They proved the following.
\begin{teo*}[Branner-Hubard]
	For a cubic polynomial $f$ with one critical point  escaping to $\infty$ and the other critical point in $\mathcal{K}_f$, the Julia set $\mathcal{J}_f$ is a Cantor set if and only if the critical component of $\mathcal{K}_f$  is aperiodic. 
\end{teo*}

Qiu and Yin in \cite{QY}, Kozlovski and van Strien in \cite{KS}, independently proved the following result, which settles the matter regarding totally disconnected in the classical dynamics setting.
\begin{teo*}[\cite{QY},\cite{KS}]
	The Julia set of a polynomial is totally disconnected if and only if each critical component of the filled Julia set is aperiodic.
\end{teo*}
As a consequence of the above Theorem, the Julia set of a complex polynomial is a Cantor set if and only if each critical component of the filled Julia set is aperiodic.

We will consider throughout this text $D_{\delta}=\{z\in \mathbb{C}; |z|<\delta\}$ and $\overline{D_{\delta}}$ the closure of $D_{\delta}$, where $\delta > 0$ is a given positive real number. The set 
\[A(\widehat{z})=\{z\in\mathbb{C};f^n(z)\rightarrow \widehat{z}\},\] is called the attractor basin of point $\widehat{z}$.

For a more detailed and comprehensive introduction to classic dynamic the reader is refereed to the books
\cite{Beardon, CG, Milnor}.

%%%%%%%%%%%%%%%%%%%
\subsection{A family of cubic polynomials}

Consider the family of cubic polynomial
\begin{equation}\label{eq4}
	f_c(z)=z^3+cz,
\end{equation}
where $c\in\mathbb{C}$. We denote here the Julia set and the Fatou set by $\mathcal{J}_c $ and $\mathcal{F}_c$ respectively. O Connectedness locus is a subset of the parameter space which consists of those parameters for which the corresponding Julia set is connected, which we denote by $\mathcal{M}$.

The finite critical points of the family \eqref{eq4}  are $z_+=\sqrt{-\frac{c}{3}}$ and $z_-=-\sqrt{-\frac{c}{3}}$. The orbit of $z_+$ is limited if, and only if, the orbit of $z_-$ is limited. In this case, the Julia set is connected if the orbit of $z_+$ is bounded (or of $z_-$).  Note that $0$ is a fixed point of $f_c$, furthermore $f'_c(0)=c$, so when $|c|<1$ we have that $0$ is a fixed point attractor. In this case the orbit of the critical points is limited and the Julia set is connected.

In \cite{Brolin} Brolin obtained the following properties associated with the family \eqref{eq1}. Brolin showed that real interval $[-3,3]\subset \mathcal{M}$. If $c\in \mathbb{R}\setminus[-3,3]$ then both critical points of $f_c$ are in the infinite attractor basin $A(\infty)$ and in this case the Julia set $\mathcal{J}_c$ is totally disconnected.	
Furthermore, if $c\in(-1,1)$ then $z_+,z_-\in A(0)$ and $\mathcal{J}_c$ is Jordan curve. Finally, for $c\in \mathbb{C}$ we have $\mathcal{J}_c\subset \overline{D}_{\delta}$, where $\delta=\sqrt{1+|c|}$.

\begin{figure}[htp!]
	\centering
	\includegraphics[scale=0.35]{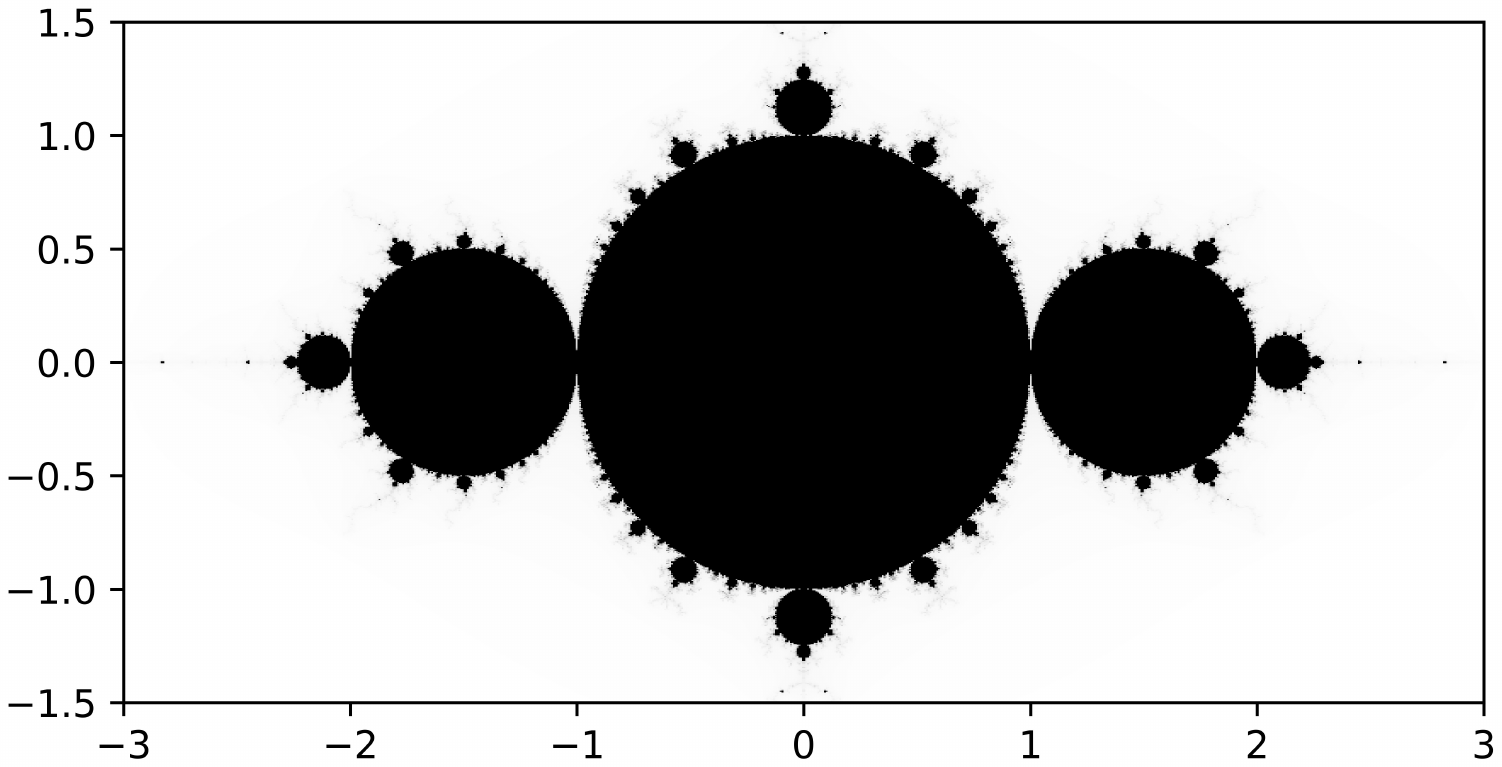}
	\caption{The Mandlbrot set of the cubic family $z^3+cz$.}
\end{figure}
%%%%%%%%%%%%%%%%%%%%%%%%%
%%%%%%%%%%%%%%%%%%%%%%%%%

\section{Green function (non-autonomous dynamic) }\label{sec:green}
A useful and important tool for investigations on non-autonomous dynamical systems is the Green's function of the basin of  infinite $\mathcal{A}_{\omega}(\infty)$. In this section, we introduce the notion of a Green's function for non-autonomous systems, together with some of its applications to random complex dynamics. We first present some preliminary results for the cubic family of polynomials \eqref{eq1} with $\omega=(c_n)\in \Omega$ and $W$ a closed disk of radius $\delta > 0$.

% (the reader can find the proof of these results in \cite{AGM}). 

\begin{lem}\label{lem:r}
	Let $0<\delta < 1$ and $\omega = (c_n) \in \Omega = \overline{D}_{\delta}^{\mathbb{N}}$. Then the filled Julia set $\mathcal{K}_{\omega}$ contains a closed
	disk $\overline{D}_r = \{ z\in \mathbb{C}; \, |z|\leq r\}$, which is an invariant domain (i.e., if $z\in \overline{D}_r$ then $f^n_{\omega}(z)\in \overline{D}_r$ for every natural number $n$), where $0< r < \sqrt{1-\delta}$. Also, when $z\in \overline{D}_r$ then $f^n_{\omega}(z) \to 0$ as $n\to\infty$.
\end{lem}

\begin{proof}
	Let $f^n_{\omega}(z) = f_{c_n} \circ \dots \circ f_{c_1}(z)$. It is sufficient to prove that if $z \in \overline{D}_r$, then $|f^n_{\omega}(z)| \leq r$ for every $n \in \mathbb{N}$. 
	
	For $n = 1$, we have 
	\begin{align*}
		|f_{c_1}(z)| &\leq |z|(|z|^2 + |c_1|) \\
		&< r (r^2 + \delta) \leq r.
	\end{align*}
	Using induction on $n$, we conclude that $f^n_{\omega}(\overline{D}_r) \subset \overline{D}_r$. More precisely, if $|f^n_{\omega}(z)| \leq r$, then 
	\begin{align*}
		|f_{\omega}^{n+1}(z)| = |f_{c_{n+1}}(f^n_{\omega}(z))| &\leq |f^n_{\omega}(z)|(|f^n_{\omega}(z)|^2 + |c_{n+1}|) \\
		&< r (r^2 + \delta) \leq r.
	\end{align*}
	
	Furthermore, since $|z| < \sqrt{1-\delta}$, it follows that $|z|^2 + \delta < 1$. Set $\gamma = |z|^2 + \delta$. Then, $|f_{c_1}(z)| < \gamma |z|$, and by induction on $n$, we have 
	\[
	|f_{\omega}^n(z)| < \gamma^n |z|.
	\]
	This completes the proof.
\end{proof}

\begin{lem}[Escape criterion]\label{lem:R}
	When $|z| > R_{\delta} := \sqrt{1+\delta}$ then there exists $\beta > 1$ such that $|f^n_{\omega}(z)| > |z| \beta^n$ for all $n\in \mathbb{N}$.
\end{lem}

\begin{proof}
	Since $|z| > R_{\delta}$, it follows that for every $n \in \mathbb{N}$, $|z|^2 - |c_n| > 1$. Hence, there exists $\beta > 1$ such that $|z|^2 - |c_n| > \beta$. We compute:
	\begin{align*}
		|f^1_{\omega}(z)| &\geq |z| (|z|^2 - |c_1|) \\
		&> |z| \beta, \\
		|f^2_{\omega}(z)| &\geq |f^1_{\omega}(z)| \left(|f^1_{\omega}(z)|^2 - |c_2| \right) \\
		&> |z| \beta^2.
	\end{align*}
	By an elementary induction argument, we obtain $|f^n_{\omega}(z)| > |z| \beta^n$ for all $n \in \mathbb{N}$.
\end{proof}

	The Green's function for the basin of infinity, measures the rate at which points escape to infinity under iteration by the rational map. Formally, it is defined as:

\begin{definition}
	Let $D$ be a proper sub-domain of $\widehat{\mathbb{C}}$. A Green's function for 
	$D$ is a map $g_D : D \times D \to (-1,1]$, such that for each $w \in D$:
	
	{\rm(1)} $g_D( - , w)$ is harmonic (for definition of harmonic functions see \cite{TR}) on $D \setminus \{w\}$ and bounded outside of every neighbourhood of $w$. Let us recall that a real-valued function $f$ on a domain $D$ of the complex plane $\mathbb{C}$ that satisfy the Laplace's equation $(\triangle f=0)$ on $D$ is said to be harmonic in $D$.
	
	{\rm (2)} $\lim_{z\to w} g_D(z,w) = \infty$ where 
	\[
	g_D(z,w) = \begin{cases}
		\log|z| + O(1) & w=\infty\\
		-\log|z-w| + O(1) & w\ne\infty.
	\end{cases}
	\]
	
	{\rm (3)} $\lim_{z\to\xi} g(z,w) = 0$ for nearly every $\xi\in\partial D$. Here ``nearly every''  means everywhere outside of some polar set.
\end{definition} 
For more details on Green's function and its properties in classic dynamic see for instance \cite{Beardon, Milnor}. Now, we shall define the non-autonomous Green's function and we prove some of its properties. These properties are mostly generalizations of their autonomous case.

Let us recall the well-known formula (see for instance \cite{Milnor}) for the autonomous case: for the map $f_c(z) = z^3+c z$ and its basin of infinity $\mathcal{A}_c(\infty)$, the Green's function with a pole at infinity is given by:
\[
g_c(z) = \lim_{n\to\infty} \frac{1}{3^n} \log|f^n_c(z)|.
\]
The following theorem is given in \cite[Theorem 4.2]{BB} and \cite[Theorem 2.1]{FS} for a general family of non-autonomous polynomials. Here give a reformulation for the case of cubic polynomials \eqref{eq1} and give a sketch of proof to verge the importance of it.
\begin{teo_green*}
	Let $W$ be a bounded Borel subset of $\mathbb{C}$, and $\Omega=W^{\mathbb{N}}$. Let $\mu$ be a Borel probability measure on $W$ , and $\mathbb{P}$, the product distribution on  $\Omega$  generated by $\mu$. For every $\omega\in\Omega$  the following limit exists: 
	\begin{eqnarray*}
		g_{\omega} : \mathcal{A}_{\omega}(\infty) &\to& \mathbb{R} \\
		z &\mapsto & \lim_{n\to\infty} \frac{1}{3^n}\log|f^n_{\omega}(z)|.
	\end{eqnarray*}
	The function $g_{\omega}$ is the Green's function on $\mathcal{A}_{\omega}(\infty)$ with pole at infinity.
	Putting $g_{\omega}$ equal to zero on the complement of $\mathcal{A}_{\omega}(\infty)$, this function extends continuously to the whole plane.
\end{teo_green*}
\begin{proof}
	We prove this theorem in four steps. In the first step, we show the existence of the limit. Then we prove that the function $g_{\omega}$ is harmonic, that $g_{\omega}(z)=\log|z| + O(1)$ when $z\rightarrow\infty$ and $g_{\omega}(z)\rightarrow 0$ when $z\to\partial\mathcal{A}_{\omega}(\infty)$.
	\\
	Step 1: Let $R>R_{\delta}$ and $c_n\in W$. By Lemma \ref{lem:R}, we have
	\[
	|f_{c_n}(z)| > |z| ~\text{for}~|z| > R.
	\]
	\begin{cla*}\label{af1}
		\[
		\frac{|z|^3}{R^2} \leq |z^3+c_n z|.
		\]
	\end{cla*}
	In fact, note that 
	\begin{eqnarray*}
		|c_n|<\delta &\Rightarrow& 1-\frac{|c_n|}{|z|^2}>1-\frac{|\delta|}{|z|^2}.\\
		|z|>R &\Rightarrow& 1-\frac{\delta}{|z|^2}>1-\frac{|\delta|}{R^2}.\\
		R>R_{\delta}&\Rightarrow& 1-\frac{\delta}{R^2}>\frac{1}{R^2}.
	\end{eqnarray*}
	Then, 
	\begin{eqnarray*}
		|z^3+c_nz| &>& |z|^3-|c_n||z|=|z|^3\left(1-\frac{|c_n|}{|z|^2}\right)\\
		&>& |z|^3\left(1-\frac{\delta}{|z|^2}\right)>|z|^3\left(1-\frac{\delta}{R^2}\right)\\
		&\geq &\frac{|z|^3}{R^2}.
	\end{eqnarray*}
Put
	\[
	U_k:=(f^k_{\omega})^{-1}(\Delta_R),
	\]
	where $\Delta_R = \mathbb{C}\setminus \overline{\mathbb{D}}_R$. Note that we have $U_k\subset U_{k+1}\subset \mathcal{A}_{\omega}(\infty)$. In fact, if $z\in U_k$ then $f^k_{\omega}(z)\in \Delta_R$ i.e., $|f^k_{\omega}(z)| > R$ and using the assumption stated above, one can conclude that 
	\[|f^{k+1}_{\omega}(z)|=|f_{c_{k+1}}(f^k_{\omega}(z))| > R.\] Therefore, $z\in U_{k+1}$. Also, it is evident that $\mathcal{A}_{\omega}(\infty) = \bigcup_{k=1}^{\infty} U_k$. Now, for $z\in U_n$, set
	\[
	g_n(z) := \frac{1}{3^n}\log \frac{|f^n_{\omega}(z)|}{R}.
	\]
	We claim that $g_n(z) \leq g_{n+1}(z)$. Actually, 
	\begin{eqnarray*}
		g_n(z) \leq g_{n+1}(z) &\Leftrightarrow& \log\frac{|f^n_{\omega}(z)|}{R} \leq \frac{1}{3} \log\frac{|f^{n+1}_{\omega}(z)|}{R}\\
		&\Leftrightarrow& \log |f^n_{\omega}(z)^3|-3\log R \leq \log |f^{n+1}_{\omega}(z)|-\log R\\
		&\Leftrightarrow& \log\frac{|f^n_{\omega}(z)|^3}{R^2} \leq \log(|f^n_{\omega}(z)^3+c_{n+1} f^n_{\omega}(z)|).
	\end{eqnarray*}
	The last inequality is valid by the  above affirmation since $z\in U_k$. Therefore, $\big(g_n\big)_{n=1}^{\infty}$ is a locally bounded increasing sequence of harmonic functions convergent to a function $G(z)$ which can be defined on $\mathcal{A}_{\omega}(\infty)$. We have:
	\begin{eqnarray*}
		G(z) &=& \lim_{n\to\infty} g_n(z) = \lim_{n\to\infty} \frac{1}{3^n}\log \frac{|f^n_{\omega}(z)|}{R}\\
		&=& \lim_{n\to \infty} \frac{\log |f^n_{\omega}(z)| - \log R}{3^n} \\
		&=& \lim_{n\to\infty} \frac{1}{3^n}\log|f^n_{\omega}(z)| = g_{\omega}(z).
	\end{eqnarray*}
	Step 2: The proof follows using Harnack's Convergence Theorem (see 
	\cite[Theorem 1.3.9]{TR}).\\
	Step 3: Since the set $W$ is bounded so for every $0<\epsilon <1$, there exists $R_{\epsilon} > 0$ such that for all $|z| > R_{\epsilon}$ and all $c_n\in W$ we have
	\[
	(1-\epsilon) |z|^3 \leq |f_{c_n}(z)| \leq (1+\epsilon) |z|^3.
	\] 
	We assume $R_{\epsilon}$ large enough so that for all $z\in\mathbb{C}$ with $|z| > R_{\epsilon}$ we also have $|f_{c_n}(z)| > |z| > R_{\epsilon}$, for all $c_n\in W$. Therefore, $\forall n\in\mathbb{N}$ using an elementary induction argument on $n$:
	\[
	(1-\epsilon)^{\frac{3^{n+1}-1}{2}} |z|^{3^n} \leq |f^n_{\omega}(z)| \leq (1+\epsilon)^{{\frac{3^{n+1}-1}{2}}} |z|^{3^n}.
	\]
	Now, taking logarithm from both sides we have
	\[
	\frac{3^{n+1}-1}{2\times3^n}\log(1-\epsilon) + \log|z| \leq \frac{1}{3^n} \log|f^n_{\omega}(z)| \leq \frac{3^{n+1}-1}{2\times3^n}\log(1+\epsilon) + \log|z|.
	\]
	When $n\to\infty$ we get
	\[
	\frac{3}{2}\log(1-\epsilon) \leq g_{\omega}(z)-\log|z|\leq \frac{3}{2}\log(1+\epsilon).
	\]
	Thereupon, when $z\to\infty$ one can conclude that 
	\[
	g_{\omega}(z)=\log|z| + O(1).
	\]
	Step 4: Let $\sigma$ denotes the shift map. Precisely speaking, if $\omega =(c_1, c_2,\dots)\in\Omega$ be a sequence of parameters then $\sigma(\omega) = (c_2,c_3,\dots)$. Let $\omega^m:=\sigma^m(\omega)=(c_m,c_{m+1},\dots)$. We have
	\[
	g_{\omega}(z) = \lim_{n\to\infty}\frac{1}{3^n}\log|f^n_{\omega}(z)| = \frac{1}{3^m}\lim_{n\to\infty}\frac{1}{3^{n-m}}\log|f^{n-m}_{\omega^m}(f^m_{\omega}(z))|.
	\]
	Choose $R>0$ large enough such that for all $\omega\in\Omega$, the set  $\{z\in\mathbb{C}; |z| >R\}$ is contained in $\mathcal{A}_{\omega}(\infty)$. Let $g$ any cubic polynomial, it is evident that for each $n\in\mathbb{N}$ and for all $z\in\mathbb{C}$ with $|z| < R$, there exists $\lambda > 1$ such that $|g^n(z)| < \lambda^{3^n}$. Also a same affirmation is valid for the random case, this means that for any $\omega\in\Omega$ we can say that $|f_{\omega}^n(z)| < \lambda^{3^n}$. From this we deduce that there exists a constant $\lambda_R$ independent of $k,m$ and $\omega$ such that for $|z| \leq R$ the following holds
	\[
	\frac{1}{3^{k-m}} \log|f^{k-m}_{\omega^m}(z)| < \lambda_R.
	\]
	Evidently, $g_{\omega}(z) < \lambda_R / 3^m$ on the set $V_m:=\{ z\in\mathbb{C}; |f^m_{\omega}(z) \leq R\}$. The sequence $(V_m)$ is a decreasing sequence of sets containing the Julia set. Therefore, $\lim\limits_{z\to\partial\mathcal{A}_{\omega}} g_{\omega}(z) = 0$.
\end{proof}
\begin{coro}\label{cor:3omega}
	If $\omega=(c_n)$ then $g_{\sigma(\omega)}(f_{\omega}(z)) =3 g_{\omega}(z)$
\end{coro}
\begin{proof} \begin{eqnarray*}
		g_{\sigma(\omega)}(f_{\omega}(z)) &=& \lim_{n\to\infty}\frac{1}{3^n}\log | f^n_{\sigma(\omega)}(f_{\omega}(z)| \\
		&=& \lim_{n\to\infty} \frac{3}{3^{n+1}}\log | f_{c_{n+1}}\circ\dots\circ f_{c_2}(f_{\omega}(z))| \\
		&=& 3\lim_{n\to\infty} \frac{1}{3^{n+1}} \log |f^{n+1}_{(c_{n+1})}(z)|\\
		&=& 3 g_{\omega}(z).
	\end{eqnarray*}
\end{proof}
%%%%%%%%%%%%%%%%%%%%%%%%%
%%%%%%%%%%%%%%%%%%%%%%%%%
\section{Proofs of the Main Results}\label{sec:Top}
We devote this section to investigate some facts on topological properties of the Julia set $\mathcal{J}_{\omega}$ with $\omega\in \Omega$ for some $\delta > 0$. 

The critical points of $f_{\omega}$ are $z_+=\sqrt{\frac{-c_1}{3}}$ and $z_-=-\sqrt{\frac{-c_1}{3}}$.
Observe that, if the orbit of $z_{+}$ is bounded, then the orbit of the critical point $z_{-} $ is also bounded.
In fact, it is enough to note that $f_{\omega}(z_+)=-f_{\omega}(z_-)$, so by a simple induction argument we have  $f^n_{\omega}(z_+)=-f^n_{\omega}(z_-)$. Thus we have a dichotomy about the orbit of the critical points, either both orbits are limited or both converge to $\infty$.

%
%
%
%\begin{rem}
%	\item[(I)] There exists a sequence $\nu =(c_n)\in \Omega$, such that its coresponded  Julia set $\mathcal{J}_{\nu}$ is disconnected but it is not totally dis connected. in other words, $\mathcal{D} \not\subset \mathcal{T}$.
%	\item[(II)] {\color{red} T é aberto? Achamos que não.
%		Dado qualquer aberto em Omega existe uma seq tal que J é desconexo mas não é tot desconexo.}
%	
%	
%\end{rem}
%\begin{proof}
	%	Let $[a, b] \subset(0,1)$. We set $c_{n_k}=-q_k$ and $c_n=0$ for $n \neq n_k$, where the sequences $\left(n_k\right)$ in $\mathbb{N}$ and $\left(q_k\right)$ in the open interval $(1,2)$ are chosen inductively as follows. We take $n_1:=1$ and $q_1 \in(1,2)$ such that $f_{c_1}([a, b]) \subset(-1,0)$. Then $f_{c_1}(0)=-q_1$. If $n_k$ and $q_k$ have already been constructed, we choose $n_{k+1}>n_k+1$ so large that $\left(f_{c_{n_{k+1}-1}} \circ \cdots \circ f_{c_{n_k+1}}\right)\left(-q_k\right)>4$. Then we take $q_{k+1} \in(1,2)$ so that $f_{c_{n_{k+1}}}\left(F_{n_{k+1}-1}([a, b])\right) \subset(-1,0)$, and thus $f_{c_{n_{k+1}}}(0)=-q_{k+1}$. Since $\left|c_n\right|<2$ we have that $\Delta_2$ is an invariant domain which implies $\mathcal{C}_{\left(c_n\right)} \subset \mathcal{A}_{\left(c_n\right)}(\infty)$. Furthermore, we see that $[a, b] \subset \mathcal{K}_{\left(c_n\right)}$. This means that the Julia set cannot be totally disconnected but it is disconnected by Theorem 1.1.
%\end{proof}
%\begin{teo}
%	Suppose that $\omega\in\Omega$ with $\delta > 3$. The set $\mathcal{T}$ defined by \eqref{eq:T} is dense in $\Omega$.
%\end{teo}
\begin{proof}[{\bf Proof of Theorem \ref{teo: T dense}}]
	Let $\omega_0 = (c_{n,0})\in \Omega$. The idea of the proof is to define a sequence of sequences in $\Omega$ that converges to $\omega_0$, whose corresponding Julia set is totally connected. Let $c\in D_{\delta}\setminus \mathcal{M}$ be a point in the coordinates axes, where $\mathcal{M}$ is the connectedness locus of family of cubic polynomials $z^3+cz$. Define
	\[
	\omega_m = (c_{n,m}) = \begin{cases}
		c_{n,0} & n = 1, \dots , m\\
		c & n>m.
	\end{cases}
	\]
	It is evident that  $c_{n,m} \to c_{n,0}$ as $m\to\infty$. Also using affirmations in \S\ref{sec:classic} and the fact that $\delta > 3$, we know that as  $c\in D_{\delta}\setminus \mathcal{M}$ so the Julia set $\mathcal{J}_c$ is totally disconnected (see \cite[Theorem 13.1]{Brolin}). We have
	\[
	\mathcal{J}_{\omega_m} = (f_{c_{m,0}} \circ \dots \circ f_{c_{1,0}})^{-1}(\mathcal{J}_c). 
	\]
	In fact, this happens due to definition of $\omega_m$ and \cite[Lemma 1.1]{Alves-Salarinoghabi}. Also, since $\mathcal{J}_c$ is totally disconnected and $f_{c_{m,0}} \circ \dots \circ f_{c_{1,0}}$ is  continuous open surjection then $\mathcal{J}_{\omega_m}$ is also totally disconnected. Actually, this is well known fact that for a continuous open surjection $g:X\to Y$ between topological spaces, the preimage of a totally disconnected set $T\subset Y$ is also totally disconnected. 
\end{proof}
%\begin{teo}
%The set $\mathcal{D}_{\infty}$ defined by \eqref{eq:Dinf} has empty interior.
%\end{teo}
%\begin{proof}
%
%\end{proof}

%%%%%%%%%%%%%%%%
%\begin{teo}\label{teo:D}
%	The set $\mathcal{D}$ given in \eqref{eq:D} is an open and dense subset of $\Omega$ provided that $\delta>3$.
%\end{teo}
\begin{proof}[{\bf Proof of Theorem \ref{teo:D}}]
	First, let demonstrate that the set $\mathcal{D}$ is open. For this reason, let $(c_{n,0})$ be a point of $\mathcal{D}$. There exists a critical point $z_0\in\mathcal{C}_{(c_{n,0})}$ such that 
	\[f^j_{(c_{n,0})}(z_0)=f_{c_{j,0}}\circ\dots\circ f_{c_{1,0}}(z_0)\in \mathcal{C}_{j}\]
	for some $j\in\mathbb{N}$ and also $f^n_{(c_{n,0})}(z_0) \to \infty$ as $n\to \infty$. This means that there exists a natural number $N>j$ such that
	\begin{eqnarray}\label{eq:f^Nj}
		\left| f_{c_{N,0}}\circ\dots\circ f_{c_{j+1,0}}\left(f^j_{(c_{n,0})}(z_0)\right)\right| > R_{\delta}. 
	\end{eqnarray}
	The recent composition depends continuously on $\left(c_{j+1,0},\dots,c_{N,0}\right)$, therefore there exists neighborhoods $U_{j+1},\dots , U_N$ of $c_{j+1,0},\dots,c_{N,0}$ respectively where  $U:=U_{j+1}\times\dots \times U_N$ is a neighborhood of $\left(c_{j+1,0},\dots,c_{N,0}\right)$ such that \eqref{eq:f^Nj} happens for all $(c_{j+1},\dots ,c_N)\in U$.  Set 
	\[\mathcal{U}:= D_{\delta}^j \times U\times D_{\delta}^{\mathbb{N}}.\]
	This is a neighborhood of $(c_{n,0})$ with respect to the product topology of $D_{\delta}^{\mathbb{N}}$. Now we shall prove that $\mathcal{U}\subset \mathcal{D}$. Let $(c_n)\in\mathcal{U}$. Since all polynomials belong to the ring $\mathbb{C}[z]$ so there exists $\psi\in\mathbb{C}$ witch satisfies in the following equation.
	\[
	f_{c_j}\circ\dots\circ f_{c_1}(z) = f_{c_{j,0}}\circ\dots\circ f_{c_{1,0}}(z_0).
	\]
	We have $(c_{j+1},\dots,c_N)\in U$ thus $|f^N_{(c_n)}(\psi)| > R_{\delta}$ and hence $(c_n)\in\mathcal{D}$.
	
	Finally, it remains to prove that $\mathcal{D}$ is dense in $\Omega$. Let $(c_{n,0})\in\mathcal{D}$. Consider the sequence of sequences $\big((c_{n,m})\big)$ defined as follows:
	\[
	c_{n,m} = \begin{cases}
		c_{n,0} & n=1,\dots,m\\
		c & n>m
	\end{cases}
	\]
	where $c\in \Omega \setminus \mathcal{M}$ and $\mathcal{M}$ is the connectedness locus of family of complex polynomials $f_c$ (in classic iteration). Evidently
	\[
	\lim_{m\to\infty} c_{n,m} = c_{n,0},
	\] 
	and the Julia set of $f_c$, denoted by $\mathcal{J}_c$, is disconnected as $c\notin \mathcal{M}$. Also one can prove that
	\[
	\mathcal{J}_{(c_{n,m})} = \left( f_{c_{m,0}} \circ \dots \circ f_{c_{1,0}}\right)^{-1}(\mathcal{J}_c).
	\]
	Therefore, $\mathcal{J}_{(c_{n,m})}$ is also disconnected and thus $(c_{n,m})\in\mathcal{D}$. 
\end{proof}
%
%\begin{lem}
%Let $K$ be compact and connected, $c\in \mathbb{C}$, $w_+$ and $w_-$ the critical values of $f_c$ and let $D_{\infty}$ be the unbounded component of $\widehat{\mathbb{C}}\setminus K$. Then $f^{-1}_{c}(K)$ is connected if and only if $w_+,w_-\notin D_{\infty}$. If $f^{-1}_c(K)$ is disconnected, then $f^{-1}_c(K)$ contains exactly three components.
%\end{lem}
%\begin{proof}If $f^{-1}_{c}(K)$ is connected then each component of $f^{-1}(\mathbb{C}\setminus K)$ is simply connected, so none of them contain a critical point, which implies that $D_{\infty}$ does not contain a critical value.
%
%If $W=f^{-1}_c(K)$ is disconnected then  $w_+,w_- \notin K$, it follows that $f_W:W\rightarrow K$ is a cover map of degree 3, therefore $f^{-1}_c(K)$ contains exactly three components. For more details see \cite[Lemma 5.7.2]{Beardon}. 
%
%\end{proof}

%\subsection{Totally disconnectedness of Julia set}
It is evident that there exists a large enough number $R > R_{\delta}=\sqrt{1+\delta} $ such that for every $\omega \in \Omega$, we have $f_{\omega}(\Delta_R) \subset \Delta_{3R}$ and also $\Delta_R \subset \mathcal{A}_{\omega}(\infty)$, where $\Delta_R = \mathbb{C}\setminus \overline{D}_R$.
\begin{prop}\label{prop:1}
	Using notations stated above, for every $\epsilon>0$ we have:
	\[
	\left|g_{\omega}(z) - \log|z| \right| < \epsilon,\quad \forall z\in \Delta_R.
	\]
\end{prop}
\begin{proof}
	For each $z\in \Delta_R$, set $a_0(z)=\log|z|$ and $\displaystyle a_n(z) = \frac{1}{3^n}\log |f^n_{\omega}(z)|$. Thereupon,
	\begin{eqnarray*}
		a_{n+1}(z) &=& \frac{1}{3^{n+1}}\log |f^{n+1}_{\omega}(z)| = \frac{1}{3^{n+1}}\log \big|\left(f^n_{\omega}(z)\right)^3+c_{n+1} \left(f^n_{\omega}(z)\right)\big| \\
		&=&  \frac{1}{3^n}\frac{1}{3} \log\big|\left(f^n_{\omega}(z)\right)^3 \big(1+\frac{c_{n+1}}{\left(f^n_{\omega}(z)\right)^2}\big) \big| \\
		&=& \frac{1}{3^n}\left(\log |f^n_{\omega}(z)| + \frac{1}{3} \log\big|\big(1+\frac{c_{n+1}}{\left(f^n_{\omega}(z)\right)^2}\big)\big|   \right)\\
		&=& a_n(z) + \frac{1}{3^{n+1}} \log\big|\big(1+\frac{c_{n+1}}{\left(f^n_{\omega}(z)\right)^2}\big)\big|.
	\end{eqnarray*}
	Using above relation one can conclude that
	\[
	a_n(z) = a_0(z) + \sum_{i=0}^n \frac{1}{3^{i+1}} \log\big|\big(1+\frac{c_{i+1}}{\left(f^i_{\omega}(z)\right)^2}\big)\big|.
	\]
	Therefore, when $n\to \infty$ we have
	\[
	g_{\omega}(z)-a_0(z) =  g_{\omega}(z)-\log|z|    =  \sum_{n=0}^{\infty} \frac{1}{3^{n+1}} \log\big|\big(1+\frac{c_{n+1}}{\left(f^n_{\omega}(z)\right)^2}\big)\big|.
	\]
	We can choose $R>R_{\delta}$ large enough such that $\left|g_{\omega}(z) - \log|z| \right| < \epsilon$.
\end{proof}

\begin{rem}
	Note that it is evident that for every $\omega\in\Omega$ the filled Julia set $\mathcal{K}_{\omega}$ is contained in the closed disk $\overline{D}_R$ where $R\geq R_{\delta} = \sqrt{1+\delta}$. 
	%Also for such $R$, if $\tilde{R}\in (R, R^3-\delta R )$ then we have $f_{\omega}^{-1}(\overline{D}_{\tilde{R}}) \subset \overline{D}_R$.
\end{rem}
%\begin{proof}
%If $z\in f_{\omega}^{-1}(\overline{D}_{\tilde{R}})$ and $z\notin \overline{D}_R$, then we have
%\begin{eqnarray*}
%R^3-\delta R > \tilde{R} \geq |f_{\omega}(z)| &> & |z|^3-|z| |c_1|\\
%& >& R^3-\delta |z|
%\end{eqnarray*}
%\end{proof}

\begin{lem}\label{lem:R tilde}
	For every $R>R_{\delta}$, fix $\tilde{R} \in (R, R^3-\delta R)$. Then for every $\omega\in\Omega$ we have
	\[
	f_{\omega}^{-1}(D_{\tilde{R}}) \subset D_R. 
	\]
	Therefore, for every $z\in\Delta_R =\mathbb{C}\setminus \overline{D}_R$,
	\[
	G_{\delta}:= \sup_{R\leq |z| \leq \tilde{R}}(g_{\omega}(z)) <\infty,
	\]
\end{lem}
\begin{proof}
	Let $z\in f_{\omega}^{-1}(D_{\tilde{R}})$. Thus, $f_{\omega}(z)\in D_{\tilde{R}}$
	\begin{eqnarray*}
		|z|^3-\delta |z| < |z|^3 - |c_1| |z| \leq |f_{\omega}(z)| < \tilde{R}.
	\end{eqnarray*}
	Solving the inequality $|z|^3-\delta |z|-\tilde{R} < 0$, we get $|z|<R_0$ where
	\[
	R_0 = \frac{1}{6}\frac{\left(108\tilde{R}+12\sqrt{(-12\delta^3+81\tilde{R}^2)}\right)^{2/3}+12\delta}{\left(108\tilde{R}+12\sqrt{(-12\delta^3+81\tilde{R}^2)}\right)^{1/3}}.
	\]
	Since $\tilde{R} < R^3-\delta R$, so by an straightforward calculation we can conclude that $R_0 < R$. Therefore $z\in D_R$. The rest is then obvious due to Proposition \ref{prop:1}.
\end{proof}
Throughout the remainder of this paper, let $R, R_{\delta}$ and $\tilde{R}$ as given in Lemma \ref{lem:R tilde}.
\begin{prop}\label{prop:2}
	For every $\delta>0$ and $\omega\in\Omega$,
	\[
	\sup_{z\in D_R} g_{\omega}(z) \leq \log R +1.
	\]
	In particular, for $\hat{z}$ a fixed critical point of $f^n_{\omega}$, the function $\omega\mapsto g_{\omega}(\hat{z})$ is bounded above and
	\[
	\sup_{\omega\in\Omega} g_{\omega}(\hat{z}) \leq \log R+1.
	\]
\end{prop}
\begin{proof}
	The proof is evident using $\epsilon = 1$ in Proposition \ref{prop:1}.
\end{proof}
\begin{dfn}
	Let $z\in D_R$ and $\omega\in\Omega$. The escape time of $z$ from $D_R$ is
	\[
	t(z,\omega)=
	\begin{cases}
		\min\{j\,;\, |f^j_{\omega}(z)|\geq R\} & z\in\mathcal{A}_{\omega}(\infty)\\
		\infty & z\in \mathcal{K}_{\omega}
	\end{cases}
	\]
\end{dfn}
\begin{prop}\label{prop:3}
	For every $z\in\mathcal{A}_{\omega}(\infty)\cap D_R$, we have
	\[
	\frac{\log R - 1}{3^{t(z,\omega)}} \leq g_{\omega}(z) \leq \frac{\log R +1}{3^{t(z,\omega)-1}}.
	\]
\end{prop}
\begin{proof}
	According to Corollary \ref{cor:3omega} we have $g_{\sigma(\omega)}(f_{\omega}(z)) = 3 g_{\omega}(z)$. Therefore, it is evident that $f_{\omega}^{t(z,\omega)-1}(z) \in D_R$ and
	\[
	g_{\sigma^{t(z,\omega)}(\omega)}(f^{t(z,\omega)}_{\omega}(z)) = 3^{t(z,\omega)} g_{\omega}(z).
	\]
	Hence, using Proposition \ref{prop:2} we have
	\[
	g_{\sigma^{t(z,\omega)}(\omega)}(f^{t(z,\omega)}_{\omega}(z)) = 3 g_{\sigma^{t(z,\omega)-1}(\omega)}(f^{t(z,\omega)-1}_{\omega}(z)) \leq 3(\log R +1).
	\]
	On the other hand
	\[
	g_{\sigma^{t(z,\omega)}(\omega)}(f^{t(z,\omega)}_{\omega}(z)) \geq \log\left| f^{t(z,\omega)}_{\omega}(z)\right| - 1 \geq \log R -1.
	\]
	This finalize the proof of this Proposition.
\end{proof}
%Due to Proposition \ref{prop:3} we can conclude that if $\hat{z}$ be a critical point of $f^n_{\omega}$ then ...?????????
\begin{lem}\label{lem: rho}
	Suppose that $\omega\in\Omega$ and $\rho \in [R,\tilde{R}]$ then
	\[\mathcal{K}_{\omega} = \bigcap_{n\in\mathbb{N}} (f^n_{\omega})^{-1}(D_{\rho}).\] 
\end{lem}
\begin{proof}
	Let $z\in \bigcap_{n\in\mathbb{N}} (f^n_{\omega})^{-1}(D_{\rho})$. Then for every $n\in\mathbb{N}$ we have $f^n_{\omega}(z)\in D_{\rho}$. So $z\in \mathcal{K}_{\omega}$. 
	Now suppose that $z\in \mathcal{K}_{\omega}$ and $z\notin \bigcap_{n\in\mathbb{N}} (f^n_{\omega})^{-1}(D_{\rho})$. This means that there exists $n\in\mathbb{N}$ such that $|f^n_{\omega}(z)| \geq \rho \geq R$. Therefore, $f^{n+j}_{\omega}(z) \to \infty$ as $j\to \infty$, which is a contradiction.
\end{proof}
%For each $k\in\mathbb{N}$, put
%\[
%B_k = (f^k_{\omega})^{-1}(D_{\rho}).
%\]
%This is a union of pairwise disjoint topological discs $B_{k,j}$. These topological discs are being mapped by a $d_j^k$-degree map $f^k_{\omega}$ to $D_{\rho}$, where $d_j < 3$ is a non-zero real number.  In the following theorem we deal with the case $\rho = \tilde{R}$, and represent a sufficient condition for total disconnectedness of the Julia set $\mathcal{J}_{\omega}$. It is worth mentioning that, a same result is valid for the quadratic random family of polynomials $z^2+c_n$ (see \cite[Proposition 10]{totall}).
%%\begin{teo}\label{teo:at most N}
%	Let $\omega \in D_{\delta}^{\mathbb{N}}$, with $\delta > 0$. If there exists a natural number $N$ such that for finitely many $k\in\mathbb{N}$ and for each component $B_{k,j}$ of $B_k$, the degree of the map $f_{\omega}^k : B_{k,j} \to D_{\tilde{R}}$ is at most $N$. Then the Julia set $\mathcal{J}_{\omega}$ is totally disconnected.
%\end{teo}
\begin{proof}[{\bf Proof of Theorem \ref{teo:at most N}}]
	\noindent
	Let $R,\tilde{R}$ and $\rho$ are as given before. For every $\nu \in \Omega$, we observe that $\Delta_R=\widehat{\mathbb{C}}\setminus D_R$ is subset of the attracting basin of infinity $\mathcal{A}_{\nu}(\infty)$. Also, we have $f_{\nu}^{-1}(D_{\tilde{R}}) \subset D_R$ (see Lemma \ref{lem:R tilde}).
	Define the annulus 
	$Ann(0;R,\tilde{R})  = \{z\in\mathbb{C}; \, R < |z| < \tilde{R}\}$.
	Now, we partition the annulus $Ann(0;R,\tilde{R})$ into $N$ smaller concentric annuli, each with the same modulus $M=\frac{1}{2\pi}\ln\left(\frac{\tilde{R}}{R}\right)$. Similar to $Ann(0;R,\tilde{R})$, these $N$ smaller annuli are located within the intersection of $D_{\tilde{R}}$ and the attracting basins $\mathcal{A}_{\nu}(\infty)$. 
	
	To demonstrate that $\mathcal{J}_{\omega}$ is totally disconnected, we prove that its only connected components are one-point sets. By the assumption of Theorem, for an increasing sequence of positive integers $(k_m)$, and for each component $B_{k_m,j}$ of $B_{k_m}$, with $m\in\mathbb{N}$, the degree of the map $f_{\omega}^{k_m}$ on $B_{k_n,j}$ is at most $N$ for all $j$ (remember that each $B_k$ is union of pairwise disjoint topological discs $B_{k,j}$).  
	
	Select an arbitrary point $z$ from the Julia set $\mathcal{J}_{\omega}$ and identify the component $B_{k_n,j_n}$ of $B_{k_n}$ containing $z$. So, there must exist at least one of the $N$ annuli free of critical values of $f_{\omega}^{k_n}$. Let this annulus be denoted as $Ann(0;R,\tilde{R})_n$. 
	
	Now, consider the disk $D' \subset D_{\tilde{R}}$, defined by the outer boundary of $Ann(0;R,\tilde{R})_n$, and denote by $B_{k_n,j_n}'$ the component of $\left(f_{\omega}^{k_n}\right)^{-1}(D')$ that contains $z$. 
	
	The map $f_{\omega}^{k_n}: B_{k_n,j_n}' \rightarrow D'$ is proper. In fact, generally we can prove that every Polynomial function $P: \mathbb{K} \to \mathbb{K}$, where $\mathbb{K}$ is is either the field of real numbers or the field of complex numbers,  is a proper map (i.e., the pre-image of a compact set is compact). The proof of this affirmation is easy and the reader can find it in any elementary analysis textbook.
	Therefore, the pre-image of $Ann(0;R,\tilde{R})_n$ under the map $f_{\omega}^{k_n}: B_{k_n,j_n}' \rightarrow D'$, is another (topological) annulus denoted by $Ann(0;R,\tilde{R})_n'$. The restriction of $f_{\omega}^{k_n}$ to $Ann(0;R,\tilde{R})_n'$ is a covering map (remember that $Ann(0;R,\tilde{R})_n$ is free of critical value of $f^{k_n}_{\omega}$) with a degree of at most $N$, ensuring that the modulus of $Ann(0;R,\tilde{R})_n'$ is at least $M/N$.
	
	%	The point $z$ lies  in the bounded component of the complement of\linebreak $Ann(0;R,\tilde{R})_n'$. 
	
	Using Lemma \ref{lem:R tilde}, we can conclude that for any $k \geq 1$, the set $f_{\sigma^{k-1}\omega}^{-1}(D_{\tilde{R}})$ is contained within $D_R$. Consequently, for any $k$ and any $\omega \in \Omega$, each component of $\left(f_{\omega}^{k+1}\right)^{-1}\left(D_{\tilde{R}}\right)$ is contained within a component of $\left(f_{\omega}^{k}\right)^{-1}\left(D_{R}\right)$, since each such component is mapped by $f_{\omega}^{k}$ onto some component of \linebreak $f_{\sigma^{k-1}\omega}^{-1}\left(D_{\tilde{R}}\right)$. This observation remains valid if $k+1$ is replaced by any integer $m > k$.
	
	Applying this for $k = k_n$ and $m = k_{n+1}$, we again find a topological annulus $Ann(0;R,\tilde{R})_{n+1}'$ of modulus at least $M/N$ in the component of $\left(f_{\omega}^{k_{n+1}}\right)^{-1}\left(D_{\tilde{R}}\right)$ containing $z$, with $z$ lying in the bounded component of the complement of $Ann(0;R,\tilde{R})_{n+1}'$.
	
	Accordingly, $Ann(0;R,\tilde{R})_{n+1}'$ is contained within the component of \linebreak $\left(f_{\omega}^{k_n}\right)^{-1}(D_R)$ containing $z$, thereby also being within the bounded component of the complement of $Ann(0;R,\tilde{R})_n'$.
	
	Thus, we construct an infinite sequence of nested, disjoint annuli \linebreak 
	$Ann(0;R,\tilde{R})_n'$, each with a modulus of at least $M/N$, all contained within $\mathcal{A}_{\omega}(\infty)$. The point $z$ remains within the bounded component of the complement of each annulus.
	
	Let $D_r$ be the outer boundary of $B_{k_n,j_n}$. Consider the topological annulus $Ann(0;r,\tilde{R})_n$. Since it includes the sequence 
	\[Ann(0;R,\tilde{R})_1', Ann(0;R,\tilde{R})_2', \dots, Ann(0;R,\tilde{R})_n',
	\]
	each with modulus at least $M/N$, so Grötzsch's inequality implies that $Ann(0;r,\tilde{R})_n$ must have a modulus of at least $nM/N$.  Due to \cite[Theorem 2.1]{McMullen} we know that 
		Any annulus $A \subset \mathbb{C}$ of sufficiently large modulus contains an essential round annulus $B$ with $mod(A) = mod(B) + O(1)$. Here essential means $B$ separates the boundary components of $A$. Therefore,  as a result,  we can conclude that $Ann(0;r,\tilde{R})_n$  contains a geometric annulus of modulus at least $nM/N - C$ (for some constant $C$), separating the boundary components of $Ann(0;r,\tilde{R})_n$. Since the connected component of $\mathcal{K}_{\omega}$ containing $z$ is in the bounded component of the complement of $Ann(0;r,\tilde{R})$ for all $n$, the component of $\mathcal{K}_{\omega}$ containing $z$ must have an arbitrarily small diameter, indicating it consists solely of the point $z$.
		
		As $z$ was chosen arbitrarily, this shows that the Julia set is totally disconnected, thus concluding the proof.
\end{proof}

%\begin{definition}\label{def:typically fast}
%	Let $W\subset\mathbb{C}$ be a bounded Borel set with $W\subset D_{\delta}$ for some $\delta > 0$. Also let $\mu_{\delta}$ be a probability Borel measure on $W$, and $\mathbb{P}={\bigotimes}_{n=0}^{\infty}\mu_{\delta}$ the product distribution on $W^{\mathbb{N}}$ generated by $\mu_{\delta}$. Suppose that $R_{\delta}$ and $G_{\delta}$ are as given in Lemma \ref{lem:R tilde}. For $\omega\in W^{\mathbb{N}}$, let $z\in\mathbb{C}$, we say that $z$ is typically fast escaping if there exists $\gamma>0$ such that
%	\[
%	\mathbb{P}\left(A_k(z)\right) < e^{-\gamma k},
%	\]   
%	where 
%	\[
%	A_k(z) = \{\omega\in\Omega ~;~g_{\omega}(z) < \frac{G_{\delta}}{3^k}\}.
%	\]
%\end{definition}
%\begin{teo}\label{teo:Ak}
%	Let $\hat{z}\in\mathcal{C}_{(f^n_{\omega})}$ be a critical point of $f^n_{\omega}$. Set
%	\[
%	A_k := \bigcap_{\hat{z}\in\mathcal{C}_{(f^n_{\omega})}} A_k(\hat{z}). 
%	\]
%	\begin{itemize}
	%		\item[(a)] For all $i=0,\dots,k-1$, if $\sigma^i\omega\notin A_{k-i}$ then, for every connected component $B_{k,j}$ of $B_k$, the map $f^k_{\omega} : B_{k,j} \to D_{\tilde{R}}$ has degree 1. 
	%		\item[(b)] If this fails for $0\leq l\leq k-1$ indices (i.e., $\sigma^i\omega \in A_k$ for $l$ indices) then the degree of the map $f^k_{\omega}$ is at most $N=3^l$. Therefore, in both cases the Julia set $\mathcal{J}_{\omega}$ is totally disconnected. 
	%	\end{itemize}
%\end{teo}

\begin{proof}[{\bf Proof of Theorem \ref{teo:Ak}}]
	Let $\hat{z}$ be an arbitrary critical point. First note that due to Lemma \ref{lem:r}, the set $A_k(\hat{z})$ is not empty. For a fixed $k$, let $B_{k,*}$ be some component of $B_k=(f^k_{\omega})^{-1}(D_{\tilde{R}})$. Also, let  $G_{\delta} =\sup_{R\leq |z|\leq \tilde{R}} g_{\omega}(z) < \infty$. Thus, for every $\nu \in\Omega$ and $z\in D_{\tilde{R}}$ we have $g_{\nu}(z) \leq G_{\delta}$. We have the following sequence of maps
	\[
	B_{k,*}
	{\overset{ f_{\omega}} \longrightarrow }
	B_{k-1,*}
	{\overset{ f_{\sigma\omega}} \longrightarrow } 
	B_{k-2,*} \longrightarrow \dots \longrightarrow B_{k-i,*}
	{\overset{ f_{\sigma^i\omega}} \longrightarrow }
	B_{k-i-1,*}
	\longrightarrow \dots
	{\overset{ f_{\sigma^{k-1} \omega}} \longrightarrow} 
	D_{\tilde{R}}
	\]
	\begin{figure}[htp!]
		\centering
		\includegraphics[scale=0.65]{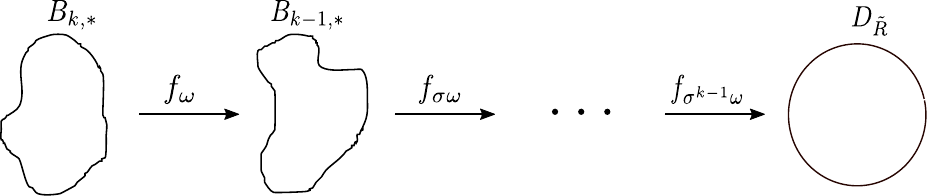}
		\caption{}\label{fig:fig1}
	\end{figure}

	If $\hat{z} \in B_{k-i,*}$ then the map $f_{\sigma^i \omega}$ has degree 3. Otherwise, using The Local Mapping Theorem, it is univalent (remember that if $f$ is an analytic map on a domain $D$, then $f'(z_0)\neq 0$ if and only if $f$ is locally univalent at $z_0$, see  \cite[Sec.~3.3]{ahlfors}). Now 
	
	(a) if      
	$\sigma^i\omega\notin A_{k-i}(\hat{z})$ then we have
	\[
	g_{\sigma^i\omega}(\hat{z}) \geq \frac{G_{\delta}}{3^{k-i}}.
	\]
	On the other hand, for all $z\in B_{k-1,*}$ we have
	\[
	g_{\sigma^i\omega}(z)=\frac{1}{3^{k-i}}g\left(f^{k-1}_{\sigma^i\omega}(z)\right) < \frac{G_{\delta}}{3^{k-i}}.
	\]
	Thereupon, $\hat{z}\notin B_{k-i,*}$ and so $f_{\sigma^i\omega}: B_{k-i,*} \to B_{k-i-1,*}$ is univalent. If this happens for all $i=0,\dots,k-1$ then the map $f^k_{\omega}$ has degree 1.  
	
	(b) If  $\sigma^i\omega \in A_k(\hat{z})$ for $l$ indices, then for these $l$ indices, the degree of $f_{\sigma^i\omega}$  can be one, two or three. And for others the degree is one. Hence, the degree of $f^k_{\omega}$ is at most $N=3^l$. Therefore in both cases, using Theorem \ref{teo:at most N}, the Julia set $\mathcal{J}_{\omega}$ is totally disconnected.
	
\end{proof}

\begin{proof}[{\bf Proof of Theorem \ref{teo: P almost every}}]
	Consider an arbitrary critical point $\hat{z}\in\mathcal{C}_{(f^n_{\omega})}$. Let $k\in\mathbb{N}$. If $\hat{z}$ is fast escaping then by Definition \ref{def:typically fast}, there exists $\gamma>0$ such that $\mathbb{P}(A_k(\hat{z})) < e^{-\gamma k}$. Consider the collection of measurable sets $\left\{A_k(\hat{z})\right\}_{k=1}^{\infty} \subset \Omega$. We have
	\[
	\sum_{k=1}^{\infty} \mathbb{P}(A_k(\hat{z})) < \sum_{k=1}^{\infty} e^{-\gamma k}= \frac{e^{-\gamma}}{1-e^{-\gamma}} < \infty.
	\] 
	Then, due to Borel-Cantelli Lemma (see \cite{Borel}) we conclude that \linebreak$\mathbb{P}\left(\bigcap_{n=1}^{\infty}\bigcup_{k=n}^{\infty}A_k(\hat{z})\right) = 0 $, or equivalently 
	\[
	\mathbb{P}\left(\left\{\omega\in\Omega; \text{  There exists infinitely many n such that } \omega\in A_n(\hat{z})   \right\}\right) = 0,
	\]
	i.e., almost surely, only finitely many $A_k(\hat{z})$'s will occur.
	For each $k\in\mathbb{N}$, since the sequence of sets $\{A_k(\hat{z})\}$ satisfies the nested property $A_{k+1}(\hat{z})\subset A_k(\hat{z})$ then $\{\mathbb{P}(A_k(\hat{z})\}$ is non-creasing. Set $A_{\infty}(\hat{z}) = \bigcap_{k=1}^{\infty} A_k(\hat{z})$. Therefore, $\mathbb{P}(A_{\infty}(\hat{z}))$ exists and is zero, since
	\[
	0 = \mathbb{P}\left(\bigcap_{n=1}^{\infty}\bigcup_{k=n}^{\infty}A_k(\hat{z})\right) = \mathbb{P}(A_{\infty}(\hat{z})).
	\]
	Thus, we can say that the sequence $(A_k(\hat{z}))$ shrinks to a zero measure set in the limit and for any fixed $\l\in\mathbb{N}$, we have
	\[
	\lim_{\l\to\infty}\mathbb{P}(\bigcup_{k=l}^{\infty}A_k(\hat{z}))=\lim_{\l\to\infty}\mathbb{P}(A_l(\hat{z})) = 0.
	\]
	Thus, for some large enough $\l$, we can guarantee that $\mathbb{P}(A_l(\hat{z}))$ is arbitrarily small. This means that there exists a positive measure set $A(\hat{z})=\Omega \setminus A_l(\hat{z})$ such that avoids all $A_k(\hat{z})$ for $k\geq \l$, i.e.,
	\[
	A(\hat{z}) \bigcap \left( \bigcup_{k=\l}^{\infty} A_k(\hat{z})\right) = \emptyset.
	\]
	Let $k \geq l$. Define the set
	\[
	\mathfrak{A}(\hat{z})=\left\{\nu\in\Omega;\,\,\sigma^i\nu \in A_{k-i}(\hat{z}) \text{ for more than } l \text{ indices } i\in\{0,\dots,k-1\} \right\}.
	\]
	{\bf{Claim:}} If $\nu\in\mathfrak{A}(\hat{z})$ then $\sigma^k\nu\notin A(\hat{z})$.\\
	Proof) Let $\nu\in\mathfrak{A}(\hat{z})$. By definition, we have $\sigma^i\nu \in A_{k-i}(\hat{z})$ for more than $l$ indices $i\in\{0,\dots,k-1\}$. Therefore, $g_{\sigma^i\nu}(\hat{z}) < \frac{G_{\delta}}{3^{k-i}}$.  Put $k-i=m$, so  $g_{\sigma^{k-m}\nu}(\hat{z}) < \frac{G_{\delta}}{3^{m}}$, i.e., $\sigma^{k-m}\nu \in A_m(\hat{z})$. Equivalently, $\sigma^k\nu \in \sigma^m\left(A_m(\hat{z})\right)$. Thus, there exists $\nu'\in A_m(\hat{z})$ such that $\sigma^k\nu = \nu'$. If $\nu=(c_n)_{n\geq 1}$, then since $\sigma^k\nu=(c_{k+n})_{n\geq 1}$ and $f^n_{\sigma^k\nu}(\hat{z})=f^{n+k}_{\nu}(\hat{z})$ so we have
	\begin{eqnarray*}
		g_{\nu'}(\hat{z})=g_{\sigma^k\nu}(\hat{z})&=&\lim_{n\to\infty}\frac{1}{3^n}\log|f^n_{\sigma^k\nu}(\hat{z})|=\lim_{n\to\infty}\frac{1}{3^n}\log|f^{n+k}_{\nu}(\hat{z})|\\ 
		&=& g_{\nu}(f^k_{\nu}(\hat{z})) < \frac{G_{\delta}}{3^m}.
	\end{eqnarray*}
	On the other hand,Since the Green's function $g_{\nu}(\hat{z})$ describes the asymptotic growth rate of the iterates of $f^n_{\nu}(\hat{z})$, the inequality suggests that the point $f^k_{\nu}(\hat{z})$ does not exhibit growth exceeding the rate $\displaystyle \frac{G_{\delta}}{3^m}$. This means that, If the value of $g_{\nu}$ at $f_{\nu}^k(\hat{z})$ is bounded above by 
	$\displaystyle\frac{G_{\delta}}{3^m}$, it suggests that the orbit of
	$\hat{z}$ under the iteration does not diverge too rapidly. More precisely, the function $g_{\nu}$ is often non-decreasing along orbits, hence
	\[
	g_{\nu}(\hat{z})<g_{\nu}(f^k_{\nu}(\hat{z})) < \frac{G_{\delta}}{3^m}.
	\]
	This implies that $\nu\in A_m(\hat{z})$. Since this happens for more than $\l$ indices $m$, the definition of the set $A(\hat{z})$ implies that $\sigma^k\nu\notin A(\hat{z})$.$\,\,\square$
	
	Backing to the proof of the Theorem, set
	\[
	\mathcal{L}(\hat{z}) = \{\nu\in\Omega; \,\,\nu\in\mathfrak{A}(\hat{z}) \text{ for all but finitely many integers } k \}.
	\]
	Using above observation, it follows that $\mathbb{P}(\mathcal{L}(\hat{z}))=0$. Now, suppose that $\nu\notin\mathcal{L}(\hat{z})$. Therefore, for infinitely many positive integers $k$, we have $\sigma^i\nu\notin A_{k-i}(\hat{z})$ for all but at most $\ell$ indices $i\in\{0,\dots,k-1\}$ thus due to item (b) of Theorem \ref{teo:Ak}, we have $\mathcal{J}_{\nu}$ is totally disconnected.
	
\end{proof}

\begin{proof}[{\bf Proof of Theorem \ref{teo7}}]
Following the notations stated before, let $\Delta_R = \{z\in\widehat{\mathbb{C}}; |z| > R > R_{\delta}\}$ be an invariant domain with $\Delta_R \cap \mathfrak{D} = \emptyset$. The complement of $\mathfrak{D}$ (witch we denote it by $\mathfrak{D}^c$ ) is a closed subset of $\mathcal{A}_{\omega}(\infty)$ and on this $f^k_{\omega}$ converges uniformly to $\infty$ as $k\to \infty$. Therefore, there exists $k'\in\mathbb{N}$ such that for all $k>k'$ we have $f^k_{\omega}(\mathfrak{D}^c) \subset \Delta_R$. The map $f^k_{\omega}(z)$ is a $3^k$-degree polynomial and has $3^k$ local inverse functions we denote them by $\mathtt{f}^k_j(z)$ for $j=1,\dots, 3^k$. Put $\mathfrak{D}_{k,j}:= \mathtt{f}^k_j(\mathfrak{D})$. These sets are disjoint and simply connected domains.
\item[{\bf Claim 1:}] $\mathcal{J}_{\omega} \subset \bigcup_{j=1}^{3^k} \mathfrak{D}_{k,j}$ for every $k\geq k'$. 
\item[Proof:] If $z\in \mathcal{J}_{\omega}$ then $f^k_{\omega}(z)\in \mathfrak{D}$ and consequently, there exists $1\leq j \leq 3^k$ such that $z\in \mathtt{f}^k_j(\mathfrak{D}).\quad \square$

\item[{\bf Claim 2:}] $\mathfrak{D}_{k,j} \subset \mathfrak{D}$ for $k\geq k'$.

\item[Proof:] 

Let $z\in \mathfrak{D}_{k,j} = \mathtt{f}^k_{j}(\mathfrak{D})$. Then $f^k_{\omega}(z)\in \mathfrak{D}$. Now if $z\notin \mathfrak{D}$ then $z\in\mathfrak{D}^c$ so $f^k_{\omega}(z)\in f^k_{\omega}(\mathfrak{D}^c) \subset \Delta_R$. Therefore, $\Delta_R \cap \mathfrak{D} \ne \emptyset$ which is a contradiction. $\quad \square$

\item[{\bf Claim 3:}] The family $\mathfrak{N} = \{\mathtt{f}^k_j : j= 1,\dots, 3^k\},$ is normal in $\mathfrak{D}$.

\item[Proof:] This is trivial using Montel Theorem (see \cite[Theorem 3.7]{Milnor}) considering the fact that $\mathtt{f}^k_j$ ($j=1,\dots, 3^k$) are local inverse functions of $f^k_{\omega}$. $\quad \square$

To prove that $\mathcal{J}_{\omega}$ is totally disconnected, let $E$ be a component of $\mathcal{J}_{\omega}$. We aim to show that $E$ is a single point. Due to claim (1), for every $k\in\mathbb{N}$, there exists $j_k\in\{1,\dots,3^k\}$ such that, $E \subset E_k:= \mathfrak{D}_{k,j_k}$. We have
\[
\mathfrak{N} \ni \mathtt{f}^k_{j_k}(\mathfrak{D}) = \mathfrak{D}_{k,j_k} = E_k \supset E.
\] 
According to the claim (3) as $\mathfrak{N}$ is a normal family in $\mathfrak{D}$ so there exists a sub-sequence $\left( \mathtt{f}^{k_{\ell}}_{j_{k_{\ell}}} \right)_{{\ell}}$ of  $\left( \mathtt{f}^{k}_{j_{k}} \right)_{k}$ such that $\mathtt{f}^{k_{\ell}}_{j_{k_{\ell}}}$ tends, locally uniformly in $\mathfrak{D}$, to a function $\phi$. If we demonstrate that $\phi$ is a constant function on $\mathfrak{D}$ then we can conclude that for every compact subset $K$ of $\mathfrak{D}$ we have:
\[
\lim_{\ell\to\infty}diam\left(\mathtt{f}^{k_{\ell}}_{j_{k_{\ell}}}(K)\right)=0,
\] 
where $diam\left(\mathtt{f}^{k_{\ell}}_{j_{k_{\ell}}}(K)\right) = \sup\{|w_1-w_2|, \,\, w_1,w_2 \in \mathtt{f}^{k_{\ell}}_{j_{k_{\ell}}}(K)\}$. This implies that $E$ is a point. Equivalently, $\mathcal{J}_{\omega}$ is totally disconnected. So, it rests to prove the following affirmation. But First we recall a well-known Hurwitz Theorem (see \cite{Zalcman}): 
\begin{teo*}[\cite{Zalcman}]
	Let $\{g_n\}$ be a sequence of holomorphic functions in a domain $U$ converging uniformly continuous toward $g$. Assume that the functions $g_n$ are injective, then $g$ is either constant or injective.
\end{teo*}
\item[{\bf Claim 4:}] $\phi$ is a constant function on $\mathfrak{D}$.
\item[Proof:] Suppose that $\phi$ is not constant on $\mathfrak{D}$, so it must be injective due to the above Theorem. For each $\ell\in\mathbb{N}$, consider the map $\mathtt{f}^{k_{\ell}}_{j_{k_{\ell}}}$ which is a local inverse of $f^{k_{\ell}}_{\omega} = f_{c_{k_{\ell}}} \circ \dots \circ f_{c_1}$. Therefore, 
%there exists $m_{\ell}\in\mathbb{N}$ such that 
$\mathtt{f}^{k_{\ell}}_{j_{k_{\ell}}}\circ f^{k_{\ell}}_{\omega}(z) = z$ on $E_{k_{\ell}}$. Now consider $\zeta\in E$. Due to Claim (1), we can assume that the sequence $\left(f^{k_{\ell}}_{\omega}(\zeta)\right)$ tends to an element of $\mathfrak{D}$, denoted by $\xi$, therefore we have $\phi(\xi) = \zeta \in \mathcal{J}_{\omega}$. On the other hand, using \cite[Lemma 1.1]{Alves-Salarinoghabi} we know that 
\[
f^k_{\omega}(\mathcal{J}_{(c_n)}) = \mathcal{J}_{(c_{n+k})}, 
\]
and also using Proposition \ref{prop:1} we have
\[
\lim_{|z|\to \infty}\big(g_{(c_{n+k})}(z) -\log|z|\big) = 0. 
\]
Hence,
\[
diam\big(f^k_{\omega}(\mathcal{J}_{(c_n)})\big) \geq 1,\quad \forall k\in\mathbb{N}.
\]
So it is possible to choose $\epsilon>0$ in such a way that the disk of center $\xi$ and radius $\epsilon$, denoted by $D(\xi,\epsilon)$, contains in $\mathfrak{D}$ and $f^k_{\omega}(\mathcal{J}_{\omega}) \not\subset D(\xi,\epsilon))$ for all $k\in\mathbb{N}$. Since $\left( \mathtt{f}^{k_{\ell}}_{j_{k_{\ell}}} \right)_{{\ell}} \to \phi$ locally uniformly on $D(\xi,\epsilon) \subset \mathcal{D}$ and $\phi$ is univalent, and also using maximum modulus principle, we conclude that $\phi(D(\xi,\epsilon))$ is simply connected. In fact, The maximum modulus principle ensures that $\phi$ achieves its maximum modulus on the boundary of $D(\xi,\epsilon)$, not in the interior. This ensures that $\phi(D(\xi,\epsilon))$ is bounded and does not extend infinitely far in the complex plane. Let 
\[
r=\inf_{z\in\partial D(\xi,\epsilon)} |\phi(z) -\phi(\xi)|.
\]
Because $\phi$ is continuous and injective, so $r>0$. Hence, there exists a neighborhood $W$ of $\phi(\xi)$ with radius $r/2$ such that $W\subset \phi(D(\xi,\epsilon))$. Now, considering the claim (2), there exists $k''\geq k'$ such that for all $k_{\ell} \geq k''$, $W \subset  \mathtt{f}^{k_{\ell}}_{j_{k_{\ell}}}\left(D(\xi,\epsilon)\right)$. From \cite[Theorem 2]{RB} we obtain that 
\[
\exists k'''\geq k'',\,\text{s.t.}\,\mathcal{J}_{\omega}\subset \left(f^k_{\omega}\right)^{-1}\left(f^k_{\omega}(W)\right) \quad \forall k\geq k'''.
\]
Finally, choose $\ell$ big enough that $k_{\ell}\geq k'''$. We have 
\[
\mathcal{J}_{\omega}\subset \left(f^{k_{\ell}}_{\omega}\right)^{-1}\left(f^{k_{\ell}}_{\omega}\left(\mathtt{f}^{k_{\ell}}_{j_{k_{\ell}}}\big(D(\xi,\epsilon)\big)\right)\right)
= \left(f^{k_{\ell}}_{\omega}\right)^{-1}\big(D(\xi,\epsilon)\big).
\]
Consequently, $f^{k_{\ell}}_{\omega}(\mathcal{J})\subset D(\xi,\epsilon)$ which is a contradiction.
\end{proof}

\begin{rem}\label{rem5}
	Observe that, by hypothesis of Theorem \ref{teo7}, we have $\mathcal{C}_{\left(f^n_{\omega}\right)}\subset \mathcal{A}_{\omega}(\infty)$. But this condition
	is neither sufficient nor necessary for $\mathcal{J}_{\omega}$ to be totally disconnected. We will show this in the next example. 
	On the other hand, a necessary condition for the Julia set $\mathcal{J}_{\omega}$ to be totally
	disconnected is that $\mathcal{C}_{\left(f^n_{\omega}\right)}\cap \mathcal{A}_{\omega}(\infty)$ is an infinite set. Otherwise there exists a positive integer $k$ such
	that the Julia set of the sequence $(F_n)_{n\geq k}$ is connected using  Theorem \CJS. But then $\mathcal{J}_{\omega}$
	has at most $2^{k-1}$ components.
\end{rem}

\subsection*{Acknowledgments}  
The first author was partially supported by  FAPEMIG APQ-02375-21, RED-00133-21 and  RED-00133-21.	The third author is supported by FAPEMIG post-doctoral scholarship with process number BPD-00761-22. 
%%%%%%%%%%%%%%%%%%%%%%%%%%%%%%%%%%%%%%%
\subsection*{Conflict of interest}
On behalf of all authors, the corresponding author states that there is no conflict of interest.

%\begin{teo}\label{teo: P almost every}
%	Let $W$, $\mu_{\delta}$ and $\mathbb{P}$ be as given in Definition \ref{def:typically fast}. If every critical point of $f^n_{\omega}$ is fast scaping (i.e., $\omega \in \mathfrak{F}$, as given in \ref{eq:F}) then the assumptions of Theorem \eqref{teo:at most N} are satisfied for $\mathbb{P}$?almost every $\omega\in W^{\mathbb{N}}$. Therefore, for $\mathbb{P}$?almost every $\omega\in W^{\mathbb{N}}$ 
%	the Julia set $\mathcal{J}_{\omega}$ is totally disconnected, i.e., for $\mathbb{P}$-almost every $\omega \in W^{\mathbb{N}}$, $\mathfrak{F} \subset \mathcal{T}$.
%\end{teo}

%%%%%%%%%%%%%%%%%%%%%%%%%
%%%%%%%%%%%%%%%%%%%%%%%%%

\end{document}